\documentclass[a4paper,12pt]{amsart}
\usepackage{amsmath}
\usepackage{amsthm}
\usepackage[latin1]{inputenc}
\usepackage{amsfonts}
\usepackage{amssymb}
\usepackage{geometry}
\usepackage{xcolor}
\newtheorem{thm}{Theorem}[section]
\newtheorem{cor}[thm]{Corollary}
\newtheorem{lem}[thm]{Lemma}

\newtheorem{defn}[thm]{Definition}
\newtheorem{rem}[thm]{Remark}

\numberwithin{equation}{section}

\def\R{\mathbb R}

\def\b{b}

\def\varkappa{m}

\newcommand{\mand}{\quad\text{and}\quad}

\title[Some examples of equivalence]{Some examples of equivalent rearrangement-invariant quasi-norms defined via $f^*$ or $f^{**}$}
\author{Leo R. Ya. Doktorski}
\address{Fraunhofer Institute of Optronics, System Technologies and Image Exploitation IOSB, Department Object Recognition, Gutleuthausstr. 1, 76275 Ettlingen, Germany.}
\email{doktorskileo@gmail.com or leo.doktorski@iosb.fraunhofer.de}

\author{Pedro Fern\'andez-Mart\'inez} 
\address{Departamento de Matem\'aticas,
Facultad de Matem\'aticas, Universidad de Murcia, Campus de
Espinardo, 30071 Espinardo (Murcia), Spain}
\email{pedrofdz@um.es}

\author{Teresa Signes} 
\address{Departamento de Matem\'aticas,
Facultad de Matem\'aticas, Universidad de Murcia, Campus de
Espinardo, 30071 Espinardo (Murcia), Spain}
\email{tmsignes@um.es}

\date{\today}
\begin{document}

\begin{abstract}
We consider Lorentz-Karamata spaces, small and grand Lorentz-Karamata spaces, and the so-called $\mathcal{L}$, $\mathcal{R}$,
$\mathcal{LL}$, $\mathcal{RL}$, $\mathcal{RL}$, and $\mathcal{RR}$ spaces. The quasi-norms for a function $f$ in each of these spaces can be
defined via the non-increasing rearrangement $f^*$ or via the maximal function $f^{**}$. We
investigate when these quasi-norms are
equivalent. Most of the proofs
are based on Hardy-type inequalities. As application we demonstrate how our general results can be used to
establish interpolation formulae for the grand and small Lorentz-Karamata spaces.
\end{abstract}
\maketitle

\section{Introduction}\label{introduction}

Let $(\Omega,\mu)$ be a totally $\sigma$-finite measure space with a non-atomic measure $\mu$ and let  $\mathcal{M}(\Omega,\mu)$ be the set of all $\mu$-measurable functions on $\Omega$. Associated to each function $f\in\mathcal{M}(\Omega,\mu)$ we consider $f^*$, the non-increasing rearrangement of $f$, and $f^{**}(t):=\frac{1}{t}\int_0^tf^*(u)du$, $t>0$, the maximal function of $f^*$
(see e.g. \cite{BS_1988}). One of the possibilities to define rearrangement-invariant quasi-Banach
function spaces on $\Omega$ is to use expressions for the quasi-norms in terms of $f^*$ or $f^{**}$. For
example, if $0<p,q\leq\infty$ and $b$ is a slowly varying function on $(0,\infty)$ (see Definition \ref{def1}  below),
the Lorentz-Karamata spaces $L_{p,q;b}$ and $L_{(p,q;b)}$ are defined as the sets of all $f\in\mathcal{M}(\Omega,\mu)$ such
that
\begin{equation}\label{ec1f}
\|f\|_{p,q;b}:=\big\|t^{\frac{1}{p}-\frac{1}{q}}b(t)f^{*}(t)\big\|_{q,(0,\infty)}<\infty
\end{equation}
or
\begin{equation}\label{ec2f}
\|f\|_{(p,q;b)}:=\big\|t^{\frac{1}{p}-\frac{1}{q}}b(t)f^{**}(t)\big\|_{q,(0,\infty)}<\infty,
\end{equation}
respectively. A natural question to ask is how this spaces are related. In \cite{OP_1999,NE_2002,GOT_2005} and \cite[Theorem 3.15]{P_2021} it was proved, using weighted Hardy inequalities, that 
\begin{equation}\label{ec1}
L_{p,q;b}=L_{(p,q;b)}\quad \mbox{if  }\quad 1<p\leq \infty\ \ \mbox{and}\ \ 0<q\leq \infty.\end{equation}
Previous equality means that the spaces are equal as linear spaces and their quasi-norms are equivalent. 


An important feature of identity \eqref{ec1} is that it can be seen from the point of view of interpolation theory. 
Recall that J. Peetre's well-known formula for the $K$-functional of a function $f$ in $L_1+L_\infty$ asserts that 
\begin{equation}\label{eK}
K(t,f;L_1,L_\infty)=\int_0^t f^*(\tau)\,d\tau=tf^{**}(t)
\end{equation}
(see \cite[Theorem V.1.6]{BS_1988}) and the interpolation spaces $(L_1,L_\infty)_{\theta,q;b}$ are defined as the set of all $f\in L_1+L_\infty$ such that 
$$ \|f\|_{\theta,q;b} := \big \| t^{-\theta-\frac{1}{q}} b(t) K(t,f;L_1,L_\infty) \big \|_{q,(0,\infty)} < \infty,$$
where $0\leq \theta\leq 1$, $0<q\leq \infty$, $b$ is a slowly varying function. 
Then 
$$(L_1,L_\infty)_{1-\frac{1}{p},q;b}=L_{(p,q;b)}.$$ Due to \eqref{ec1}, we have that
$$(L_1,L_\infty)_{1-\frac{1}{p},q;b}=L_{(p,q;b)}=L_{p,q;b},\qquad \mbox{if}\ 1<p<\infty,\ 0<q\leq \infty.$$
The aim of this paper is to prove similar identities for the limiting and extremal cons\-tructions $\mathcal{R}$, $\mathcal{L}$, 
$\mathcal{LL}$, $\mathcal{RL}$, $\mathcal{RL}$ and $\mathcal{RR}$ (see precise definitions in \S \ref{interpolation methods}). Specifically, these limiting  and extremal methods applied to the couple $(L_1,L_\infty)$ give spaces with quasi-norms expressed in terms of $f^{**}$. These expressions written with $f^*$ instead of $f^{**}$ produce quasi-norms that define, a priori, different quasi-Banach function spaces. Our goal  is to prove that the spaces defined via $f^{*}$ or $f^{**}$ are equal and the quasi-norms are equivalent. See \S \ref{section6} for the precise statements. In particular our results can be applied to grand and small Lebesgue spaces \cite{FK_2004}, as well as grand and small Lorentz-Karamata spaces, see Theorem \ref{teo63} and \ref{teo64}.

We illustrate the utility of our results by including an application to the interpolation of the couple $\big(L^{p_0),q_0,r_0}_{\b_0},L_\infty\big)$ (see Corollary \ref{cor72}). Other examples can be seen in \cite{D_2020B,FS_2021A,FS_2021B,FS_2021C,D_2021A,DFS_2022,DFS_2022B}.

The  organization of the paper is the following: Section 2 contains the basic tools: the properties of slowly varying functions and new Hardy-type inequalities. In section 3, we describe the interpolation methods we shall work with, namely $(A_0,A_1)_{\theta,q;\b}$, the ${\mathcal R}$ and ${\mathcal L}$ limiting interpolation spaces and the $\mathcal{LL}$, $\mathcal{RL}$, $\mathcal{RL}$ and $\mathcal{RR}$ extremal interpolation spaces. We also define the function spaces under consideration through quasi-norms expressed in terms of $f^*$ or $f^{**}$. 
In Section 4 we prove the main lemmas. Section 5 is devoted to the interpolation formulae for the couples $(L_m,L_\infty)$ and $(L_{m,\infty},L_\infty)$ and in Section 6 we deduce the particular case $m=1$.
Finally, Section 7 contains the applications.

\section{Preliminaries}\label{preliminaries}
Throughout the paper, we write $X\subset Y$ for two (quasi-) normed spaces $X$ and $Y$ to indicate that $X$ is continuously embedded in $Y$. We write $X=Y$ if $X\subset Y$ and $Y\subset X$. In this case, $X$ and $Y$ are equal as sets and as linear spaces and their quasi-norms are equivalent, and we say that the spaces $X$ and   $Y$ are identical or equal.

 For $f$ and $g$ being
positive functions, we write $f\lesssim g$ if $f\leq C g$, where the constant $C$ is independent of all
significant quantities. Two positive functions $f$ and $g$ are considered equivalent $f\approx g$ if $f\lesssim g$ and $g\lesssim f$. We adopt the conventions $1/\infty=0$ and $1/0=\infty$. The abbreviation LHS(*) and RHS(*)
will be used for the left and right hand side of the relation (*), respectively. By $\chi_{(a,b)}$ we denote the
characteristic function on an interval $(a,b)$. We write $f\!\uparrow$ ($f\!\downarrow$) if the positive function $f$ is non-decreasing (non-increasing).
Moreover, $\|*\|_{q,(a,b)}$ is the usual (quasi-) norm in the Lebesgue space $L_q$ on the interval $(a,b)$ ($0<q\leq\infty$, $0\leq a<b\leq\infty$).

\subsection{Slowly varying functions}
In this subsection we summarize some of the pro\-per\-ties of slowly varying functions which will be required later. For more details, we refer to e.g. \cite{FS_2012,FS_2015,GOT_2005,P_2021}.

\begin{defn}\label{def1}
A positive Lebesgue measurable function $\b$, $0\not\equiv\b\not\equiv\infty$,
is said to be \textit{slowly varying} on $(0,\infty)$ (notation $\b\in SV$) if, for each $\varepsilon>0$, the function $t \leadsto t^\varepsilon\b(t)$  is  equivalent to an increasing function while $ t \leadsto t^{-\varepsilon}\b(t)$ is equivalent to a decreasing function. 
\end{defn}

Examples of $SV$-functions include powers of logarithms,
$$\ell^\alpha(t)=(1+|\log t|)^\alpha(t),\quad t>0,\quad \alpha\in\R,$$ ``broken" logarithmic functions of \cite{EO_2000}, 
reiterated logarithms $(\ell\circ\ldots\circ\ell(t))^\alpha,\; t>0$, $\alpha\in\R$, and also the family of functions  $\exp(|\log t|^\alpha),\; t >0$ for  $\alpha \in (0,1)$.

\begin{lem}\label{lemma2}
Let $\b, \b_1, \b_2\in SV$, $\lambda\in\R$, $\alpha>0$, $0<q\leq \infty$ and $t\in(0,\infty)$. 
\begin{itemize}
\item[(i)]  Then $\b^\lambda\in SV$, $\b(1/t)\in SV$, $\b(t^\alpha \b_1(t))\in SV$, and $\b_1\b_2\in SV$.
\item[(ii)] If $f\approx g$, then $\b\circ f\approx \b\circ g$.
\item[(iii)] 
$\big\|u^{\alpha-\frac{1}{q}}\b(u)\big\|_{q,(0,t)}\approx t^\alpha\b(t)$ and
$\big\|u^{-\alpha-\frac{1}{q}}\b(u)\big\|_{q,(t,\infty)}\approx t^{-\alpha}\b(t).$
\item[(iv)] The functions $\big\|u^{-\frac{1}{q}}\b(u)\big\|_{q,(0,t)}$ and 
$\big\|u^{-\frac{1}{q}}\b(u)\big\|_{q,(t,\infty)}$ (if exist) belong to $SV$. Moreover,
$$\b(t) \lesssim \big\|u^{-\frac{1}{q}} \b(u)\big\|_{q,(0,t)} \mand
\b(t) \lesssim \big\|u^{-\frac{1}{q}} \b(u)\big\|_{q,(t,\infty)}.$$
\item[(v)] $\big\|u^{\lambda-\frac{1}{q}}\b(u)\big\|_{q,(t/2,t)}\approx\big\|u^{\lambda-\frac{1}{q}}\b(u)\big\|_{q,(t,2t)}\approx t^\lambda\b(t).$
\end{itemize}
\end{lem}

Along the paper, we will often use these properties without explicit reference every time. The following simple lemma extends the above inequalities.

\begin{lem}\label{lemma3}
Let $\lambda\in\R$, $0<q\leq \infty$ and $\b\in SV$. Then for all non-negative non-increasing Lebesgue measurable functions $f$ on $(0,\infty)$ and for all $t>0$,
$$t^{\lambda}\b(t)f(t)\lesssim\big\|u^{\lambda-\frac{1}{q}}\b(u)f(u)\big\|_{q,(0,t)},$$ 
$$t^{\lambda}\b(t)f(t)\lesssim\big\|u^{\lambda-\frac{1}{q}}\b(u)f(u)\big\|_{q,(\frac{t}{2},\infty)}$$
and 
$$ t^{\lambda}\b(t)f(2t)\lesssim\big\|u^{\lambda-\frac{1}{q}}\b(u)f(u)\big\|_{q,(t,2t)}.$$
\end{lem}
\begin{proof}
Since $f$ is non-increasing, by Lemma \ref{lemma2} (v), we have 
\begin{align*}
t^{\lambda}\b(t)f(t)&\approx\|u^{\lambda-\frac{1}{q}}\b(u)\|_{q,(\frac{t}{2},t)}f(t)\leq \|u^{\lambda-\frac{1}{q}}\b(u)f(u)\|_{q,(\frac{t}{2},t)}\leq\|u^{\lambda-\frac{1}{q}}\b(u)f(u)\|_{q,(0,t)}.
\end{align*}
The other two estimates can be proved similarly.
\end{proof}

\subsection{Hardy-type inequalities}
The following Hardy-type inequalities and their corollaries will be essential parts of our later arguments.

\begin{lem}\label{teHardy1}\cite[Lemma 2.7]{GOT_2005}
Let $ 1\leq q\leq \infty$ and $\b\in SV$. The inequality
$$\Bigl\|u^{\alpha-\frac{1}{q}}\b(u)\int_0^u
f(\tau)\,d\tau\Bigr\|_{q,(0,\infty)}\lesssim
\|u^{\alpha+1-\frac{1}{q}}\b(u)f(u)\|_{q,(0,\infty)}$$
holds for all non-negative Lebesgue measurable functions $f$ on $(0,\infty)$ if and only if $\alpha<0$.
\end{lem}
Next lemma recovers Corollary 2.9 from \cite{GOT_2005}, and additionally completes it with the cases $0<T<S<\infty$ and $\alpha\leq -1$.
%
%
\begin{lem}\label{lemma5}
Let $\b\in SV$ and $\alpha<0$. Then,
\begin{itemize}
\item[i)] If $0<q<1$, the inequality
\begin{equation}\label{Hardy1TS}
\Bigl\|u^{\alpha-\frac{1}{q}}\b(u)\int_0^u
f(\tau)\,d\tau\Bigr\|_{q,(T,S)}\lesssim T^\alpha\b(T)\int_0^T\!\!f(s)\,ds+
\big\|u^{\alpha+1-\frac{1}{q}}\b(u)f(u)\big\|_{q,(T,S)}
\end{equation}
holds for all $0\leq T<S\leq \infty$ and all non-negative non-increasing Lebesgue measurable functions $f$ on $(0,\infty)$.
\item[ii)] If $1\leq q\leq \infty$, \eqref{Hardy1TS} holds for all $0\leq T<S\leq \infty$ and all non-negative  Lebesgue measurable functions $f$ on $(0,\infty)$.
\end{itemize}
\end{lem}
\begin{proof}
We can assume that $f\neq 0$. 

Case $0<q<1$. Denote $G(u)=\int_0^u f(\tau)d\tau$. It is easy to see that 
\begin{equation}\label{ec22}
G^q(u)=q\int_0^u G^{q-1}(\tau)f(\tau)d\tau,\qquad u>0.
\end{equation}
and hence
\begin{align*}
\Bigl\|u^{\alpha-\frac{1}{q}}\b(u)\int_0^u
f(\tau)\,d\tau\Bigr\|^q_{q,(T,S)}
&=\int_T^S u^{\alpha q}\b^q(u)G^q(u)\frac{du}{u}\\
&=q\int_T^S u^{\alpha q}\b^q(u)\int_0^u G^{q-1}(\tau)f(\tau)d\tau\frac{du}{u}.
\end{align*}
Thus, by Fubini's theorem, \eqref{ec22} and Lemma \ref{lemma2}, we deduce 
\begin{align*}
\Bigl\|u^{\alpha-\frac{1}{q}}\b(u)&\int_0^u
f(\tau)\,d\tau\Bigr\|^q_{q,(T,S)}\\
&\approx G^q(T)\int_T^S u^{\alpha q}\b^q(u)\frac{du}{u}+\int_T^S \Big(\int_\tau^S u^{\alpha q}\b^q(u)\frac{du}{u}\Big) G^{q-1}(\tau)f(\tau)d\tau\\
&\leq G^q(T)\int_T^\infty u^{\alpha q}\b^q(u)\frac{du}{u}+\int_T^S \Big(\int_\tau^\infty u^{\alpha q}\b^q(u)\frac{du}{u}\Big) G^{q-1}(\tau)f(\tau)d\tau\\
&\approx T^{\alpha q}\b^q(T)G^q(T)+\int_T^S \tau^{\alpha q}\b^q(\tau)G^{q-1}(\tau)f(\tau)d\tau.
\end{align*}
Furthermore, the hypothesis $f\!\downarrow$ implies $G(u)\geq f(u)\int_0^ud\tau=uf(u)$ and, since $q-1<0$, we have $G^{q-1}(u)\leq (uf(u))^{q-1}$, $u>0$. Multiplying by $f(u)$, we obtain that the inequality
$$G^{q-1}(u)f(u)\leq u^{q-1}f^q(u)$$
 holds for all $u>0$. Consequently, 
$$
\Bigl\|u^{\alpha-\frac{1}{q}}\b(u)\int_0^u
f(\tau)\,d\tau\Bigr\|^q_{q,(T,S)}\lesssim T^{\alpha q}\b^q(T)G^q(T)+\int_T^S \tau^{q(\alpha+1)-1}\b^q(\tau)f^q(\tau)d\tau.$$
Case $1\leq q\leq \infty$. This time we have 
\begin{align*}
\Bigl\|u^{\alpha-\frac{1}{q}}\b(u)&\int_0^u
f(\tau)\,d\tau\Bigr\|_{q,(T,S)}\\&\approx\big\|u^{\alpha-\frac{1}{q}}\b(u)\big\|_{q,(T,S)}\int_0^T
f(\tau)\,d\tau+\Bigl\|u^{\alpha-\frac{1}{q}}\b(u)\int_T^u
f(\tau)\,d\tau\Bigr\|_{q,(T,S)}\\
&\lesssim T^{\alpha-\frac{1}{q}}\b(T)\int_0^T
f(\tau)\,d\tau+\Bigl\|u^{\alpha-\frac{1}{q}}\b(u)\int_T^u
f(\tau)\,d\tau\Bigr\|_{q,(T,S)}.
\end{align*}
Moreover, if we apply Hardy-type inequality of Lemma \ref{teHardy1} we obtain that the second term of the last expression
is bounded by $\|u^{\alpha+1-\frac{1}{q}}\b(u)f(u)\|^q_{q,(T,S)}$ . Indeed,
\begin{align*}
\Bigl\|u^{\alpha-\frac{1}{q}}\b(u)\int_T^u
f(\tau)\,d\tau\Bigr\|_{q,(T,S)}&\leq\Bigl\|u^{\alpha-\frac{1}{q}}\b(u)\int_0^u
f(\tau)\chi_{(T,S)}(\tau)\,d\tau\Bigr\|_{q,(0,\infty)}\\
&\lesssim\Bigl\|u^{\alpha+1-\frac{1}{q}}\b(u)
f(u)\chi_{(T,S)}(u)\Bigr\|_{q,(0,\infty)}\\
&=\Bigl\|u^{\alpha+1-\frac{1}{q}}\b(u)
f(u)\Bigr\|_{q,(T,S)}.
\end{align*}
This completes de proof.
\end{proof}

\begin{cor}\label{corollary6}
Let $\alpha<0$, $0<q,r\leq \infty$, and $a, \b\in SV$. Then, for all non-negative non-increasing Lebesgue measurable functions $f$ on $(0,\infty)$
\begin{align}
\biggl\|t^{-\frac{1}{r}}\b (t)\Bigl\|u^{\alpha-\frac{1}{q}}a(u)&\int_0^u
f(\tau)\,d\tau\Bigr\|_{q,(t,\infty)}\biggr\|_{r,(0,\infty)}\label{ec24}\\
&\approx\Bigl\|t^{-\frac{1}{r}}\b (t)\bigl\|u^{\alpha+1-\frac{1}{q}}a(u)
f(u)\bigr\|_{q,(t,\infty)}\Bigr\|_{r,(0,\infty)}.\nonumber
\end{align}
\end{cor}
\begin{proof}
By hypothesis $f\!\downarrow$, then $\int_0^u f(\tau)d\tau\geq uf(u)$ and the estimate ``$\geq$'' holds. Next we prove the reverse estimate. Lemma \ref{lemma5} yields
$$LHS(\ref{ec24})\lesssim \Bigl\|t^{\alpha-\frac{1}{r}}\b(t)a(t)\int_0^t
f(\tau)\,d\tau\Bigr\|_{r,(0,\infty)}+RHS(\ref{ec24}).$$ Moreover, by Hardy-type inequality of Lemma \ref{teHardy1}, Lemma \ref{lemma3} and a change of variables we obtain that the first term of the previous sum is bounded by the second one. That is,
\begin{align*}
\Big\|t^{\alpha-\frac{1}{r}}\b(t)a(t)\int_0^t
f(\tau)\,d\tau\Big\|_{r,(0,\infty)}&\lesssim \big\|t^{\alpha+1-\frac{1}{r}}\b(t)a(t)f(t)\big\|_{r,(0,\infty)}\\
&\lesssim \Bigl\|t^{-\frac{1}{r}}\b (t)\bigl\|u^{\alpha+1-\frac{1}{q}}a(u)
f(u)\bigr\|_{q,(\frac{t}{2},\infty)}\Bigr\|_{r,(0,\infty)}\\
&\approx RHS(\ref{ec24}).
\end{align*}
This completes the proof.
\end{proof}
\begin{cor}\label{corollary7}
Let $\alpha<0$, $0<q,r,s\leq \infty$, and $a, \b, c\in SV$. Then, for all non-negative non-increasing Lebesgue measurable functions $f$ on $(0,\infty)$
\begin{align}
\Biggl\|t^{-\frac{1}{s}}c(t)\biggl\|u^{-\frac{1}{r}}\b&(u)\Bigl\|v^{\alpha-\frac{1}{q}}a(v)\int_0^v
f(\tau)\,d\tau\Bigr\|_{q,(u,\infty)}\biggr\|_{r,(t,\infty)}\Biggr\|_{s,(0,\infty)}\label{ec25}\\
&\approx \biggl\|t^{-\frac{1}{s}}c(t)\Big\|u^{-\frac{1}{r}}\b(u)\big\|v^{\alpha+1-\frac{1}{q}}a(v)
f(v)\big\|_{q,(u,\infty)}\Bigr\|_{r,(t,\infty)}\biggr\|_{s,(0,\infty)}.\nonumber
\end{align}
\end{cor}
\begin{proof}
Since $f\downarrow$ the estimate ``$\gtrsim$'' holds. Let us prove  the reverse estimate ``$\lesssim$''. By Lemma \ref{lemma5}, we have that $LHS(\ref{ec25})$ is bounded by 
\begin{align*}
\biggl\|t^{-\frac{1}{s}}c(t)\Bigl\|u^{-\frac{1}{r}}\b(u)&\Bigl(u^{\alpha}a(u)\int_0^u
f(\tau)\,d\tau+\big\|v^{\alpha+1-\frac{1}{q}}a(v)f(v)\big\|_{q,(u,\infty)}\Bigr)\Bigr\|_{r,(t,\infty)}\biggr\|_{s,(0,\infty)}\\
&\lesssim \biggl\|t^{-\frac{1}{s}}c(t)\Bigl\|u^{\alpha-\frac{1}{r}}\b(u)a(u)\int_0^u
f(\tau)\,d\tau\Bigr\|_{r,(t,\infty)}\biggr\|_{s,(0,\infty)}+RHS(\ref{ec25}).
\end{align*}
Then, to finish the proof, it suffices to show that
$$I:=\biggl\|t^{-\frac{1}{s}}c(t)\Bigl\|u^{\alpha-\frac{1}{r}}\b(u)a(u)\int_0^u
f(\tau)\,d\tau\Bigr\|_{r,(t,\infty)}\biggr\|_{s,(0,\infty)}\lesssim RHS(\ref{ec25}).$$
Using Corollary \ref{corollary6}, Lemma \ref{lemma3}, and the simple change of variables $y=u/2$ and $x=t/2$, we conclude that
\begin{align*}
I&\approx \Bigl\|t^{-\frac{1}{s}} c(t)\bigl\|u^{\alpha+1-\frac{1}{r}}\b(u)a(u)
f(u)\bigr\|_{r,(t,\infty)}\Bigr\|_{s,(0,\infty)}\\
&\lesssim\biggl\|t^{-\frac{1}{s}} c(t)\Bigl\|u^{-\frac{1}{r}}\b(u)\big\|v^{\alpha+1-\frac{1}{q}}a(v)
f(v)\big\|_{q,(\frac{u}{2},\infty)}\Bigr\|_{r,(t,\infty)}\biggr\|_{s,(0,\infty)}\\
&\approx\biggl\|t^{-\frac{1}{s}} c(t)\Bigl\|y^{-\frac{1}{r}}\b(y)\big\|v^{\alpha+1-\frac{1}{q}}a(v)
f(v)\big\|_{q,(y,\infty)}\Bigr\|_{r,(\frac{t}{2},\infty)}\biggr\|_{s,(0,\infty)}\\
&\approx\biggl\|x^{-\frac{1}{s}} c(x)\Bigl\|y^{-\frac{1}{r}}\b(y)\big\|v^{\alpha+1-\frac{1}{q}}a(v)
f(v)\big\|_{q,(y,\infty)}\Bigr\|_{r,(x,\infty)}\biggr\|_{s,(0,\infty)}=RHS(\ref{ec25}).
\end{align*}
\end{proof}

\begin{rem}
Note that in Lemma \ref{lemma5} and in Corollaries \ref{corollary6} and \ref{corollary7} the assumption ``non-increasing''  is necessary only if $0 <q< 1$.
\end{rem}

\section{Interpolation Methods and Function spaces} \label{interpolation methods}

We refer to the monographs \cite{BS_1988,BL_1976,BK_1991,KPS_1982,T_1978} for basic concepts on Interpolation Theory and Banach function spaces.

\subsection{Interpolation Methods}

Next, we collect the definitions and basic properties of the real interpolation methods defined with slowly varying functions. This should give a sufficient background to follow the paper.

In what follows $\overline{A}=(A_{0}, A_{1})$ will be a compatible (quasi-) Banach couple such that $A_0\cap A_1\neq \{0\}$. For $t>0$, the \textit{Peetre $K$-functional} is given by
 \begin{align*}
K(t,f;A_0,A_1)\equiv K(t,f)=\inf \Big \{\|f_0\|_{A_0}+t\|f_1\|_{A_1}:\ f=f_0+f_1,\ f_i\in A_i , \; i=0,1
\Big \}.
\end{align*}

In recent years, the following scale of interpolation spaces with slowly varying functions has been intensively studied.

\begin{defn}\label{defrealmethod}\cite{GOT_2005}
Let $0\leq\theta \leq 1$, $0<q\leq \infty$ and  $a\in SV$. The real interpolation space $\overline{A}_{\theta,q;a}\equiv(A_0,A_1)_{\theta,q;a}$ consists of all $f$ in $A_{0} + A_{1}$ that satisfy
$$ \|f\|_{\theta,q;a} := \big \| t^{-\theta-\frac{1}{q}} a(t) K(t,f) \big \|_{q,(0,\infty)} < \infty.$$
\end{defn}

\begin{lem}\cite[Proposition 2.5]{GOT_2005}
Let $0\leq \theta\leq 1$, $0<q\leq \infty$ and $a\in SV$. 
The space $\overline{A}_{\theta,q;a}$ is  a (quasi-) Banach space. Moreover, $ A_0\cap A_1 \hookrightarrow  \overline{A}_{\theta,q;a} \hookrightarrow A_0+A_1$ if and only if one  of the following conditions is satisfied:
\begin{itemize}
\item[(i)] $0 < \theta < 1$,  

\item[(ii)] $\theta =0$ and $\big\|t^{-\frac{1}{q}}a(t)\big\|_{q,(1,\infty)} \!<\infty$  or 

\item[(iii)] $\theta =1$, $\big\|t^{-\frac{1}{q}}a(t)\big\|_{q,(0,1)} \!<\infty$ .
\end{itemize}
If none of these conditions holds, then $\overline{A}_{\theta,q;a}=\{0\}$. 
\end{lem}

\begin{defn}\label{defLR}\cite{GOT_2005}
Let $0\leq \theta\leq 1$, $0<r,q\leq \infty$ and $a, \b\in SV$. The space $\overline{A}_{\theta,r,\b,q,a}^{\mathcal L}\equiv(A_0,A_1)_{ \theta,r,\b,q,a}^{\mathcal L}$ consists of all $f \in A_{0} + A_{1}$ for which 
$$ \|f \|_{\mathcal L;\theta,r,\b,q,a}:=
\Big\|t^{-\frac{1}{r}}\b(t) \|u^{-\theta-\frac{1}{q}} a(u) K(u,f) \|_{q,(0,t)}\Big\|_{r,(0,\infty)} < \infty. $$
Silimilarly, the space
$\overline{A}_{\theta,r,\b,q,a}^{\mathcal R}\equiv(A_0,A_1)_{\theta,r,\b,q,a}^{\mathcal R}$ consists of all $f\in A_0+A_1$ for which
$$\| f  \|_{\mathcal{R};\theta,r,\b,q,a} := \Big \|t^{-\frac{1}{r}}  \b(t) \|   u^{-\theta-\frac{1}{q}} a(u) K(u,f) \|_{q,(t,\infty)}      \Big   \|_{r,(0,\infty)}<\infty.$$
\end{defn}

These spaces naturally arise in reiteration formulae for the limiting cases $\theta=0$ or $\theta=1$. In literature (see e.g. \cite{FS_2012,D_2020B,FS_2015,D_1992,AEEK_2011,DFS_2022}) similar definitions are given alongside with properties of these spaces. As proposed in \cite{FS_2012}, we refer to these spaces as $\mathcal{L}$ and $\mathcal{R}$ limiting interpolation spaces.  The next two lemmas can be proved in the usual way.

\begin{lem} (Cf. \cite{DFS_2022})
Let $0\leq \theta\leq 1$, $0<r,q\leq \infty$ and $a, \b\in SV$. 
The space $\overline{A}_{\theta,r,\b,q,a}^{\mathcal L}$ is  a (quasi-) Banach space. Moreover, $ A_0\cap A_1 \hookrightarrow  \overline{A}_{\theta,r,\b,q,a}^{\mathcal L} \hookrightarrow A_0+A_1$ if and only if one  of the following conditions is satisfied:
\begin{itemize}
\item[(i)] $0 < \theta < 1$  and  $\big\|t^{-\frac{1}{r}}\b(t)\big\|_{r,(1,\infty)} <\infty$,  

\item[(ii)] $\theta =0$ and $ \Big \|t^{-\frac{1}{r}} \b(t)\|u^{-\frac{1}{q}}a(u)\|_{q,(1,t)}  \Big \|_{r,(1,\infty)} \!< \infty$  or 

\item[(iii)] $\theta =1$, $\big\|t^{-\frac{1}{r}}\b(t)\big\|_{r,(1,\infty)} \!<\infty$ and $\Big \|t^{-\frac{1}{r}} \b(t) \|u^{-\frac{1}{q}}a(u)\|_{q,(0,t)} \Big \|_{r,(0,1)} \!< \infty$.
\end{itemize}
If none of these conditions holds, then $\overline{A}^{\mathcal L}_{\theta,r,\b,q,a}$ is  the trivial space.
\end{lem}

\begin{lem} (Cf. \cite{DFS_2022})
Let $0\leq \theta\leq 1$, $0<r,q\leq \infty$ and $a, \b\in SV$. The space $\overline{A}_{\theta,r,\b,q,a}^{\mathcal R}$ is a (quasi-) Banach space. Moreover, 
$ A_0\cap A_1 \hookrightarrow  \overline{A}_{\theta,r,\b,q,a}^{\mathcal R} \hookrightarrow A_0+A_1$ one of the following conditions is satisfied:
\begin{enumerate}
\item[(i)] $0 < \theta < 1$  and  $\big\|t^{-\frac{1}{r}}\b(t)\big\|_{r,(0,1)} <\infty$,  

\item[(ii)] $\theta =0$, $\big\|t^{-\frac{1}{r}}\b(t)\big\|_{r,(0,1)} \!<\infty$ and $\Big \|t^{-\frac{1}{r}} \b(t) \|u^{-\frac{1}{q}}a(u)\|_{q,(t,\infty)} \Big \|_{r,(1,\infty)} \!< \infty$.

\item[(iii)] $\theta =1$ and $ \Big \|t^{-\frac{1}{r}} \b(t)\|u^{-\frac{1}{q}}a(u)\|_{q,(t,1)}  \Big \|_{r,(0,1)} \!< \infty$  or 
\end{enumerate}
Otherwise, $\overline{A}^{\mathcal R}_{\theta,r,\b,q,a}=\{0\}$.
\end{lem}

In the next definition, we introduce four more interpolation spaces. We follow \cite{FS_2021A,FS_2021B,FS_2021C,D_2021A,DFS_2022} where it has been shown that they appear in relation with the extreme reiteration results.

\begin{defn}\label{defLRR^}
Let $0\leq \theta\leq 1$, $0<s,r,q\leq \infty$ and $a, \b, c\in SV$. 
The space
$\overline{A}_{\theta,s,c,r,\b,q,a}^{\mathcal L,\mathcal L}\equiv(A_0,A_1)_{ \theta,s,c,r,\b,q,a}^{\mathcal L,\mathcal L}$ is the set of all $f\in A_0+A_1$ for which  the quasi-norm
$$ \|f \|_{\mathcal L,\mathcal L;\theta,s,c,r,\b,q,a} :=
\bigg\|t^{-\frac{1}{s}}c(t)\Big\|u^{-\frac{1}{r}}\b(u) \big\|v^{-\theta-\frac{1}{q}} a(v) K(v,f) \big\|_{q,(0,u)}\Big\|_{r,(0,t)}\bigg\|_{s,(0,\infty)}$$
is finite. The spaces $\overline{A}_{\theta,s,c,r,\b,q,a}^{\mathcal L,\mathcal R}\equiv(A_0,A_1)_{ \theta,s,c,r,\b,q,a}^{\mathcal L,\mathcal R}$, $\overline{A}_{\theta,s,c,r,\b,q,a}^{\mathcal R,\mathcal L}\equiv(A_0,A_1)_{ \theta,s,c,r,\b,q,a}^{\mathcal R,\mathcal L}$, $\overline{A}_{\theta,s,c,r,\b,q,a}^{\mathcal R,\mathcal R}\equiv(A_0,A_1)_{ \theta,s,c,r,\b,q,a}^{\mathcal R,\mathcal R}$ are defined via the quasi-norms
$$ \|f \|_{\mathcal L,\mathcal R;\theta,s,c,r,\b,q,a} :=\bigg\|t^{-\frac{1}{s}}c(t)\Big\|u^{-\frac{1}{r}}\b(u) \big\|v^{-\theta-\frac{1}{q}} a(v) K(v,f) \big\|_{q,(t,u)}\Big\|_{r,(t,\infty)}\bigg\|_{s,(0,\infty)}
,$$
$$ \|f \|_{\mathcal R,\mathcal L;\theta,s,c,r,\b,q,a} :=\bigg\|t^{-\frac{1}{s}}c(t)\Big\|u^{-\frac{1}{r}}\b(u) \big\|v^{-\theta-\frac{1}{q}} a(v) K(v,f) \big\|_{q,(u,t)}\Big\|_{r,(0,t)}\bigg\|_{s,(0,\infty)}$$
and 
$$ \|f \|_{\mathcal R,\mathcal R;\theta,s,c,r,\b,q,a} :=
\bigg\|t^{-\frac{1}{s}}c(t)\Big\|u^{-\frac{1}{r}}\b(u) \big\|v^{-\theta-\frac{1}{q}} a(v) K(v,f) \big\|_{q,(u,\infty)}\Big\|_{r,(t,\infty)}\bigg\|_{s,(0,\infty)},$$
respectively.
\end{defn}

For these spaces it is possible to formulate conditions under which they are trivial. For example, if $\|u^{-\frac{1}{r}}\b(u)\|_{r,(1,\infty)}=\infty$ then $\overline{A}_{\theta,s,c,r,\b,q,a}^{\mathcal L,\mathcal R}=\{0\}$. We leave this to the reader.  For other combinations of parameters, they are non-trivial interpolation spaces between $A_0$ and $A_1$. In what follows we will assume that no spaces under consideration are trivial. We refer to these spaces as 
$\mathcal{L}\mathcal{L}$, $\mathcal{L}\mathcal{R}$, $\mathcal{R}\mathcal{L}$ and $\mathcal{R}\mathcal{R}$ extremal interpolation spaces.

\subsection{Function spaces}
Next we define function spaces under consideration which quasi-norm are based pairwise on $f^*$ and on $f^{**}$. (The latter ones will be denote as $L_{(p,q;a)}$ and ($L$)-spaces). The principal aim of this paper is to investigate when these spaces are equal (with equivalent quasi-norms).

\begin{defn}
Let $0<p,q\leq \infty$ and $a\in SV$. The Lorentz-Karamata type spaces $L_{p,q;a}$ and $L_{(p,q;a)}$ are the sets of all $f\in\mathcal{M}(\Omega,\mu)$ such that 
\begin{equation}\label{eLK}
\|f\|_{p,q;a}:=\big\|t^{\frac{1}{p}-\frac{1}{q}}a(t)f^*(t)\big\|_{q,(0,\infty)}<\infty
\end{equation}
and 
$$\|f\|_{(p,q;a)}:=\big\|t^{\frac{1}{p}-\frac{1}{q}}a(t)f^{**}(t)\big\|_{q,(0,\infty)}<\infty,$$
respectively.
\end{defn}

The Lorentz-Karamata spaces comprises important scales of spaces. It contains e.g. the Lebesgue spaces $L_q$, Lorentz spaces $L_{p,q}$, Lorentz-Zygmund, and the generalized Lorentz-Zygmund spaces. We refer to \cite{NE_2002,GOT_2005,P_2021,FS_2014} for further information about Lorentz-Karamata spaces and to \cite{BS_1988,AFH_2020,GOT_2005,P_2021,AEEK_2011,FS_2014,EO_2000,OP_1999} for important applications in analysis.

\begin{defn} (Cf. \cite[(5.21), (5.33)]{GOT_2005}
Let $0<p,q,r\leq \infty$ and $a, \b\in SV$. The spaces $L_{p,r,\b,q,a}^{\mathcal L}$, $(L)_{p,r,\b,q,a}^{\mathcal L}$, $L_{p,r,\b,q,a}^{\mathcal R}$, $(L)_{p,r,\b,q,a}^{\mathcal R}$ are the sets of all $f\in\mathcal{M}(\Omega,\mu)$ such that 
\begin{equation}\label{eLKL}
\|f\|_{L_{p,r,\b,q,a}^{\mathcal L}}:=\Big\|t^{-\frac{1}{r}}\b(t)\big\|u^{\frac{1}{p}-\frac{1}{q}}a(u)f^*(u)\big\|_{q,(0,t)}\Big\|_{r,(0,\infty)}<\infty,
\end{equation}
$$\|f\|_{(L)_{p,r,\b,q,a}^{\mathcal L}}:=\Big\|t^{-\frac{1}{r}}\b(t)\big\|u^{\frac{1}{p}-\frac{1}{q}}a(u)f^{**}(u)\big\|_{q,(0,t)}\Big\|_{r,(0,\infty)}<\infty,$$
$$\|f\|_{L_{p,r,\b,q,a}^{\mathcal R}}:=\Big\|t^{-\frac{1}{r}}\b(t)\big\|u^{\frac{1}{p}-\frac{1}{q}}a(u)f^*(u)\big\|_{q,(t,\infty)}\Big\|_{r,(0,\infty)}<\infty,$$
or
$$\|f\|_{(L)_{p,r,\b,q,a}^{\mathcal R}}:=\Big\|t^{-\frac{1}{r}}\b(t)\big\|u^{\frac{1}{p}-\frac{1}{q}}a(u)f^{**}(u)\big\|_{q,(t,\infty)}\Big\|_{r,(0,\infty)}<\infty,$$
respectively.
\end{defn}

We will require that $\big\|t^{-\frac{1}{r}}\b(t)\big\|_{r,(1,\infty)}<\infty$ for $L_{p,r,\b,q,a}^{\mathcal{L}}$ and $(L)_{p,r,\b,q,a}^{\mathcal L}$ spaces and $\big\|t^{-\frac{1}{r}}\b(t)\big\|_{r,(0,1)}<\infty$ for $L_{p,r,\b,q,a}^{\mathcal R}$ and $(L)_{p,r,\b,q,a}^{\mathcal R}$ spaces. Otherwise the corresponding spaces consist only on the null-element. Similar definitions can be found in \cite{FS_2021C,FS_2015,EO_2000,D_2021A,D_2018A}. We refer to these spaces 
as $\mathcal{L}$ and $\mathcal{R}$ Lorentz-Karamata spaces, respectively. Note that the $\mathcal{L}$ spaces are special cases of Generalized Gamma spaces with double weights \cite{FFGKR_2018}.

In order to compare our results with those from \cite{FFGKR_2018,AFH_2020,FS_2021B,FS_2015} we introduce the grand and small Lorentz-Karamata spaces.

\begin{defn}
Let $0<p,q,r\leq \infty$ and $\b\in SV$. The small Lorentz-Karamata spaces $L^{(p,r,q}_\b$ and $(L)^{(p,q,r}_\b$ and the grand Lorentz-Karamata spaces $L^{p),q,r}_\b$ and $(L)^{p),q,r}_\b$  are the sets of all $f\in\mathcal{M}(\Omega,\mu)$ such that 
\begin{equation}\label{eL}
\|f\|_{L^{(p,q,r}_\b}:=\Big\|t^{-\frac{1}{r}}\b(t)\big\|u^{\frac{1}{p}-\frac{1}{q}}f^*(u)\big\|_{q,(0,t)}\Big\|_{r,(0,\infty)}<\infty,
\end{equation}
$$\|f\|_{(L)^{(p,q,r}_\b}:=\Big\|t^{-\frac{1}{r}}\b(t)\big\|u^{\frac{1}{p}-\frac{1}{q}}f^{**}(u)\big\|_{q,(0,t)}\Big\|_{r,(0,\infty)}<\infty,$$
$$\|f\|_{L^{p),q,r}_\b}:=\Big\|t^{-\frac{1}{r}}\b(t)\big\|u^{\frac{1}{p}-\frac{1}{q}}f^*(u)\big\|_{q,(t,\infty)}\Big\|_{r,(0,\infty)}<\infty,$$
or
$$\|f\|_{(L)^{p),q,r}_\b}:=\Big\|t^{-\frac{1}{r}}\b(t)\big\|u^{\frac{1}{p}-\frac{1}{q}}f^{**}(u)\big\|_{q,(t,\infty)}\Big\|_{r,(0,\infty)}<\infty,$$
respectively.
\end{defn}

\begin{rem}\label{remark21}
It is clear that $L^{(p,q,r}_\b=L_{p,r,\b,q,1}^{\mathcal{L}}$, $(L)^{(p,q,r}_\b=(L)_{p,r,\b,q,1}^{\mathcal{L}}$, $L^{p),q,r}_\b=L_{p,r,\b,q,1}^{\mathcal{R}}$ and $(L)^{p),q,r}_\b=(L)_{p,r,\b,q,1}^{\mathcal{R}}$.
\end{rem}


Grand and small Lebesgue and Lorentz spaces find many important applications and they have been intensive studied by different authors. See \cite{FFGKR_2018,AFH_2020,FS_2021A,FS_2021B,FS_2021C,FS_2015} and the references therein. These spaces are often defined on a bounded domain $\Omega$ in $\R^n$ with measure 1; sometimes, see \cite{AFH_2020}, they are also restricted to real valued functions. Here, we do not require that $\mu(\Omega)=1$ and neither that the functions are real-valued.

\begin{defn}
Let $0<p,q,r,s\leq \infty$ and $a, \b, c\in SV$. The spaces $L^{\mathcal L,\mathcal L}_{p,(s,c,r,\b,q,a)}$,   $L^{\mathcal L,\mathcal R}_{p,(s,c,r,\b,q,a)}$,   $L^{\mathcal R,\mathcal L}_{p,(s,c,r,\b,q,a)}$ and $L^{\mathcal R,\mathcal R}_{p,(s,c,r,\b,q,a)}$ 
are the set of all $f\in\mathcal{M}(\Omega,\mu)$ such that 
\begin{equation}\label{ecL1}
\|f \|_{L^{\mathcal L,\mathcal L}_{p,(s,c,r,\b,q,a)}} :=
\bigg\|t^{-\frac{1}{s}}c(t)\Big\|u^{-\frac{1}{r}}\b(u) \|v^{\frac{1}{p}-\frac{1}{q}} a(v) f^*(v) \|_{q,(0,u)}\Big\|_{r,(0,t)}\bigg\|_{s,(0,\infty)},
\end{equation}
\begin{equation}\label{ecL2} \|f \|_{L^{\mathcal L,\mathcal R}_{p,(s,c,r,\b,q,a)}} :=
\bigg\|t^{-\frac{1}{s}}c(t)\Big\|u^{-\frac{1}{r}}\b(u) \|v^{\frac{1}{p}-\frac{1}{q}} a(v) f^*(v) \|_{q(t,u)}\Big\|_{r,(t,\infty)}\bigg\|_{s,(0,\infty)},
\end{equation}
\begin{equation}\label{ecL3}\|f \|_{L^{\mathcal R,\mathcal L}_{p,(s,c,r,\b,q,a)}} :=
\bigg\|t^{-\frac{1}{s}}c(t)\Big\|u^{-\frac{1}{r}}\b(u) \|v^{\frac{1}{p}-\frac{1}{q}} a(v) f^*(v) \|_{q,(u,t)}\Big\|_{r,(0,t)}\bigg\|_{s,(0,\infty)},\end{equation}
\begin{equation}\label{ecL4} \|f \|_{L^{\mathcal R,\mathcal R}_{p,(s,c,r,\b,q,a)}} :=
\bigg\|t^{-\frac{1}{s}}c(t)\Big\|u^{-\frac{1}{r}}\b(u) \|v^{\frac{1}{p}-\frac{1}{q}} a(v) f^*(v) \|_{q,(u,\infty)}\Big\|_{r,(t,\infty)}\bigg\|_{s,(0,\infty)},\end{equation}
respectively. 

The spaces $(L)^{\mathcal L,\mathcal L}_{p,(s,c,r,\b,q,a)}$,  $(L)^{\mathcal L,\mathcal R}_{p,(s,c,r,\b,q,a)}$,  $(L)^{\mathcal R,\mathcal L}_{p,(s,c,r,\b,q,a)}$ and $(L)^{\mathcal R,\mathcal R}_{p,(s,c,r,\b,q,a)}$ are defined as the set of all $f\in\mathcal{M}(\Omega,\mu)$ such that \eqref{ecL1}-\eqref{ecL4} are satisfied after the change of $f^*$ by $f^{**}$, respectively.
\end{defn}

We refer to these spaces as ${\mathcal L\mathcal L}$, ${\mathcal L\mathcal R}$, ${\mathcal R\mathcal L}$ and ${\mathcal R\mathcal R}$ Lorentz-Karamata spaces. 

The next lemma follows from Lemma \ref{lemma3}. The proof is left to the reader.

\begin{lem}\label{lemma23}
Let $0<p,q,r,s\leq \infty$ and $a, \b, c\in SV$. Then
$$L_{p,(s,c,r,\b,q,a)}^{\mathcal L,\mathcal L}\subset L_{p,s,c,r,a \b}^{\mathcal L}\subset L_{p,s;a\b c},$$
$$L_{p,(s,c,r,\b,q,a)}^{\mathcal R,\mathcal R}\subset L_{p,s,c,r,a \b}^{\mathcal R}\subset L_{p,s;a\b c},$$
$$L_{p,(s,c,r,\b,q,a)}^{\mathcal L,\mathcal R}\subset L_{p,s;B},\ \mbox{where}\ B(t)=c(t)a(t)\big\|u^{-\frac{1}{r}}\b(u)\big\|_{r,(t,\infty)}$$
and 
$$L_{p,(s,c,r,\b,q,a)}^{\mathcal R,\mathcal L}\subset L_{p,s;B},\ \mbox{where}\ B(t)=c(t)a(t)\big\|u^{-\frac{1}{r}}\b(u)\big\|_{r,(0,t)}$$
In particular, 
$$L_\b^{(p,q,r}\subset L_{p,r;\b}\mand L_\b^{p),q,r}\subset L_{p,r;\b}.$$
Analogous inclusions hold if the $L$-spaces are replaced by $(L)$-spaces.
\end{lem}

\subsection{Spaces which quasi-norms are defined via $f^{**}$}

Peetre's formula \eqref{eK} allows us to characterize all $(L)$-spaces as interpolation spaces for the couple $(L_1,L_\infty)$ through the appropriate interpolation method. Indeed,

\begin{lem}\label{lemma24}
Let $1\leq p\leq \infty$, $0<q\leq \infty$, $a, \b\in SV$ and $\theta=1-\frac{1}{p}$. Then
$$L_{(p,q;a)}=(L_1,L_\infty)_{\theta,q;a}.$$
\end{lem}

\begin{lem}\label{lemma25}(See \cite[Lemmas 5.4, 5.9]{GOT_2005})
Let $1\leq p\leq \infty$, $0<q,r\leq \infty$, $a, \b\in SV$ and $\theta=1-\frac{1}{p}$. Then,
$$(L)_{p,r,\b,q,a}^{\mathcal L}=(L_1,L_\infty)^{\mathcal L}_{\theta,r,\b,q,a} \mand 
(L)_{p,r,\b,q,a}^{\mathcal R}=(L_1,L_\infty)^{\mathcal R}_{\theta,r,\b,q,a}.$$
In particular,
$$(L)^{(p,q,r}_{\b}=(L_1,L_\infty)^{\mathcal L}_{\theta,r,\b,q,1} \mand 
(L)^{p),q,r}_{\b}=(L_1,L_\infty)^{\mathcal R}_{\theta,r,\b,q,1}.$$
\end{lem}

\begin{lem}\label{lemma26}
Let $1\leq p\leq \infty$, $0<q,r,s\leq \infty$, $a, \b,c\in SV$ and $\theta=1-\frac{1}{p}$. Then
$$(L)_{p,(s,c,r,\b,q,a)}^{\mathcal L,\mathcal L}=(L_1,L_\infty)^{\mathcal L,\mathcal L}_{\theta,s,c,r,\b,q,a} ,\qquad
(L)_{p,(s,c,r,\b,q,a)}^{\mathcal R,\mathcal R}=(L_1,L_\infty)^{\mathcal R,\mathcal R}_{\theta,s,c,r,\b,q,a},$$
$$(L)_{p,(s,c,r,\b,q,a)}^{\mathcal L,\mathcal L}=(L_1,L_\infty)^{\mathcal L,\mathcal L}_{\theta,s,c,r,\b,q,a} \mand 
(L)_{p,(s,c,r,\b,q,a)}^{\mathcal R,\mathcal R}=(L_1,L_\infty)^{\mathcal R,\mathcal R}_{\theta,s,c,r,\b,q,a}.$$
\end{lem}
\section{Main Lemmas}

Recall that $(\Omega,\mu)$ denote a totally $\sigma$-finite measure space with a non-atomic measure $\mu$ and $\mathcal{M}(\Omega,\mu)$ is the set of all $\mu$-measurable functions on $\Omega$. 

Let $0<\varkappa<\infty$. Peetre's formula \eqref{eK} was been generalized by P. Kr\'ee \cite{K_1967} in the following sense. For  $f\in L_\varkappa+L_\infty$,
\begin{equation}\label{eKsigma}
K(t,f;L_\varkappa,L_\infty)\approx\Big(\int_0^{t^\varkappa} f^*(\tau)^\varkappa d\tau\Big)^{1/\varkappa}, \qquad t>0.
\end{equation}
If $f\in L_{\varkappa,\infty}+L_\infty$, then \cite{H_1970}
$$K(t,f;L_{\varkappa,\infty},L_\infty)\approx\sup_{0<\tau<t^\varkappa} \tau^{\frac{1}{\varkappa}} f^*(\tau), \qquad t>0.$$
Let (Cf. \cite{Y_1969,S_1972})
\begin{equation}\label{e7}
f^{**}_{(\varkappa)}(t):=\frac{1}{t}\Big(\int_0^{t^\varkappa} f^*(\tau)^\varkappa d\tau\Big)^{1/\varkappa}, \qquad t>0.
\end{equation}
Obviously $f^{**}_{(1)}=f^{**}$. Futhermore, $f^{**}_{(m)}$ satisfies the following properties:
\begin{lem}\label{lemf} (Cf. \cite[(R8)]{Y_1969}.)
\begin{itemize}
\item[(i)] The function $f^{**}_{(\varkappa)}$ (if exists) is non-increasing.
\vspace{2mm}
\item[(ii)] $f^{**}_{(\varkappa)}(t)\geq f^{*}(t^{\varkappa})$, for all $t>0$.
\vspace{2mm}
\item[(iii)] $(f+g)^{**}_{(\varkappa)}(t)\lesssim f^{**}_{(\varkappa)}(t)+g^{**}_{(\varkappa)}(t)$, for all $t>0$.
\end{itemize}
\end{lem}
\begin{proof}
The first assertion can be proved in the usual way. See for example \cite[Proposition 3.2]{BS_1988}.
In order to prove (ii) we use that $f^*$ is non-increasing to obtain 
$$f^{**}_{(\varkappa)}(t)\geq \frac{1}{t}f^*(t^\varkappa)\Big(\int_0^{t^\varkappa}  d\tau\Big)^{1/\varkappa}=f^*(t^\varkappa).$$
(iii) The equivalence \eqref{eKsigma} and the (quasi)-subadditivity of the $K$-functional, imply that
\begin{align*}
(f+g)^{**}_{(\varkappa)}(t)&\approx \frac{1}{t}K(t,f+g;L_\varkappa,L_\infty)\lesssim\frac{1}{t}K(t,f;L_\varkappa,L_\infty)+\frac{1}{t}K(t,g;L_\varkappa,L_\infty)\\
&\approx f^{**}_{(\varkappa)}(t)+g^{**}_{(\varkappa)}(t).
\end{align*}
This completes the proof of the lemma.
\end{proof}

The following result is a modification of \cite[Lemma 5.2]{GOT_2005} and \cite[Lemmas 8.2 and 8.3]{EO_2000} and can be proved similarly. In what follows we shall denote $\dddot{\b}(u)=\b(u^{\frac{1}{\varkappa}})$, $u>0$, $0<\varkappa<\infty$.

\begin{lem}\label{lemma15}
Let $\theta\in(0,1]$, $0<\varkappa<\infty$, $0<q\leq \infty$ and $\b\in SV$. Then, for all $f\in L_\varkappa+L_\infty$ and all $t>0$
\begin{align*}
\Big\|u^{-\theta-\frac{1}{q}}\b(u)K(u,f;L_{\varkappa},L_\infty)\Big\|_{q,(0,t)}&\approx\Big\|u^{1-\theta-\frac{1}{q}}\b(u)f^{**}_{(\varkappa)}(u)\Big\|_{q,(0,t)}\\&\approx \Big\|v^{\frac{1-\theta}{\varkappa}-\frac{1}{q}}\dddot{\b}(v)f^{*}(v)\Big\|_{q,(0,t^{\varkappa})},
\end{align*}
and for all $f\in L_{\varkappa,\infty}+L_\infty$ and all $t>0$
$$\Big\|u^{-\theta-\frac{1}{q}}\b(u)K(u,f;L_{\varkappa,\infty},L_\infty)\Big\|_{q,(0,t)}\approx \Big\|v^{\frac{1-\theta}{\varkappa}-\frac{1}{q}}\dddot{\b}(v)f^{*}(v)\Big\|_{q,(0,t^{\varkappa})}.$$
\end{lem}

In the next two lemmas we generalize, in some sense, Lemma \ref{lemma15}. Indeed, we estimate, from above and below, the quasi-norm in $L_q(T,S)$ of the functions $$u^{-\theta-1/q}\b(u)K(u,f;L_{\varkappa},L_\infty)\quad \mbox{and}\quad u^{-\theta-1/q}\b(u)K(u,f;L_{\varkappa,\infty},L_\infty),$$ for all $0\leq T<S\leq\infty$. 
\begin{lem}\label{lemma16}
Let $\theta\in(0,1]$, $0<\varkappa<\infty$, $0<q\leq \infty$ and $\b\in SV$. Then, for all $f\in L_\varkappa+L_\infty$ and all $0\leq T<S\leq\infty$,
$$\Big\|u^{-\theta-\frac{1}{q}}\b(u)K(u,f;L_{\varkappa},L_\infty)\Big\|_{q,(T,S)}\gtrsim\Big\|v^{\frac{1-\theta}{\varkappa}-\frac{1}{q}}\dddot{\b}(v)f^{*}(v)\Big\|_{q,(T^{\varkappa},S^{\varkappa})},$$
and for all $f\in L_{\varkappa,\infty}+L_\infty$ and all $0\leq T<S\leq\infty$
\begin{equation}\label{e43}
\Big\|u^{-\theta-\frac{1}{q}}\b(u)K(u,f;L_{\varkappa,\infty},L_\infty)\Big\|_{q,(T,S)}\gtrsim\Big\|v^{\frac{1-\theta}{\varkappa}-\frac{1}{q}}\dddot{\b}(v)f^{*}(v)\Big\|_{q,(T^{\varkappa},S^{\varkappa})}.
\end{equation}
\end{lem}

\begin{proof}
Let $f\in L_\varkappa+L_\infty$. By \eqref{eKsigma} and Lemma \ref{lemf} (ii) we have $K(u,f;L_\varkappa,L_\infty)\approx uf^{**}_{(\varkappa)}(u)\geq uf^*(u^\varkappa)$, $u>0$. Hence, using the change of variables $v=u^\varkappa$, it follows
\begin{align*}
\Big\|u^{-\theta-\frac{1}{q}}\b(u)K(u,f;L_{\varkappa},L_\infty)\Big\|_{q,(T,S)}&\gtrsim\Big\|u^{1-\theta-\frac{1}{q}}\b(u)f^{*}(u^\varkappa)\Big\|_{q,(T,S)}\\&\approx\Big\|v^{\frac{1-\theta}{\varkappa}-\frac{1}{q}}\dddot{\b}(v)f^{*}(v)\Big\|_{q,(T^{\varkappa},S^{\varkappa})}.
\end{align*}
Estimate \eqref{e43} can be proved similarly using that $$K(u,f;L_{\varkappa,\infty},L_\infty)\approx\sup_{0<\tau<u^\varkappa} \tau^{\frac{1}{\varkappa}} f^*(\tau)\geq uf^*(u^\varkappa).$$
The proof is completed.
\end{proof}

\begin{lem}\label{lemma17}
Let $\theta\in(0,1]$, $0<\varkappa<\infty$, $0<q\leq \infty$ and $\b\in SV$. Then, for all $f\in L_\varkappa+L_\infty$ and all $0\leq T<S\leq \infty$,
$$
\Big\|u^{-\theta-\frac{1}{q}}\b(u)K(u,f;L_{\varkappa},L_\infty)\Big\|_{q,(T,S)}\lesssim T^{1-\theta}\b(T)f^{**}_{(\varkappa)}(T)+\Big\|v^{\frac{1-\theta}{\varkappa}-\frac{1}{q}}\dddot{\b}(v)f^{*}(v)\Big\|_{q,(T^\varkappa,S^{\varkappa})},$$
and, for all $f\in L_{\varkappa,\infty}+L_\infty$ and all $0\leq T<S\leq\infty$,
\begin{align}
\Big\|u^{-\theta-\frac{1}{q}}\b(u)K(u,f;L_{\varkappa,\infty},L_\infty)\Big\|_{q,(T,S)}&\lesssim T^{1-\theta}\b(T)f^{**}_{(\varkappa)}(T)\nonumber\\ &+\Big\|v^{\frac{1-\theta}{\varkappa}-\frac{1}{q}}\dddot{\b}(v)f^{*}(v)\Big\|_{q,(T^\varkappa,S^{\varkappa})}.\label{ecP}
\end{align}
\end{lem}
\begin{proof}
Let $f\in L_{\varkappa}+L_\infty$. By \eqref{eKsigma} it holds
$$
\Big\|u^{-\theta-\frac{1}{q}}\b(u)K(u,f;L_{\varkappa},L_\infty)\Big\|_{q,(T,S)}\approx 
\bigg\|u^{-\theta-\frac{1}{q}}\b(u)\Big(\int_0^{u^\varkappa} f^*(\tau)^\varkappa d\tau\Big)^{1/\varkappa}\bigg\|_{q,(T,S)}.$$
For the last term, the suitable change of variables, Lemma \ref{lemma5} and \eqref{e7} yield that
\begin{align}
\bigg\|u^{-\theta-\frac{1}{q}}&\b(u)\Big(\int_0^{u^\varkappa} f^*(\tau)^\varkappa d\tau\Big)^{1/\varkappa}\bigg\|_{q,(T,S)}\nonumber\\
&
=\bigg\|u^{-\theta\varkappa-\frac{\varkappa}{q}}\b(u)^{\varkappa}\int_0^{u^\varkappa} f^*(\tau)^\varkappa d\tau\bigg\|_{q/\varkappa,(T,S)}^{1/\varkappa}\nonumber\\
&\approx\bigg\|v^{-\theta-\frac{\varkappa}{q}}\dddot{\b}(v)^{\varkappa}\int_0^{v} f^*(\tau)^\varkappa d\tau\bigg\|_{q/\varkappa,(T^\varkappa,S^\varkappa)}^{1/\varkappa}\nonumber\\
&\lesssim 
\bigg(T^{-\theta\varkappa}\b(T)^\varkappa\int_0^{T^\varkappa}f^{*}(\tau)^\varkappa\, d\tau+\Big\|v^{1-\theta-\frac{\varkappa}{q}}\dddot{\b}(v)^\varkappa f^{*}(v)^\varkappa\Big\|_{q/\varkappa,(T^\varkappa,S^{\varkappa})}\bigg)^{1/\varkappa}\nonumber\\
&\approx 
T^{-\theta}\b(T)\bigg(\int_0^{T^\varkappa}f^{*}(\tau)^\varkappa\, d\tau\bigg)^{1/\varkappa}+\Big\|v^{1-\theta-\frac{\varkappa}{q}}\dddot{\b}(v)^\varkappa f^{*}(v)^\varkappa\Big\|_{q/\varkappa,(T^\varkappa,S^{\varkappa})}^{1/\varkappa}\nonumber\\
&=T^{1-\theta}\b(T)f^{**}_{(\varkappa)}(T)+\Big\|v^{\frac{1-\theta}{\varkappa}-\frac{1}{q}}\dddot{\b}(v)f^{*}(v)\Big\|_{q,(T^\varkappa,S^{\varkappa})}.\label{ec89}
\end{align}
This concludes the first part of the proof.

Let $f\in L_{\varkappa,\infty}+L_\infty$ and $u>0$. Since
\begin{align}
K(u,f;L_{\varkappa,\infty},L_\infty)&\approx \sup_{0<z<u^m}z^{\frac{1}{m}}f^*(z) \approx
\sup_{0<z<u^\varkappa} f^*(z)\Big(\int_0^z d\tau\Big)^{1/\varkappa}\nonumber\\&\leq \sup_{0<z<u^\varkappa}\Big(\int_0^z f^*(\tau)^\varkappa d\tau\Big)^{1/\varkappa}=\Big(\int_0^{u^\varkappa} f^*(\tau)^\varkappa d\tau\Big)^{1/\varkappa}\label{ec99}
\end{align}
we have that 
$$\Big\|u^{-\theta-\frac{1}{q}}\b(u)K(u,f;L_{\varkappa,\infty},L_\infty)\Big\|_{q,(T,S)}\lesssim \bigg\|u^{-\theta-\frac{1}{q}}\b(u)\Big(\int_0^{u^\varkappa} f^*(\tau)^\varkappa d\tau\Big)^{1/\varkappa}\bigg\|_{q,(T,S)}.$$
Now, using \eqref{ec89} we obtain \eqref{ecP}.

\end{proof}

\section{Interpolation formulae for the couples $(L_\varkappa,L_\infty)$ and $(L_{\varkappa,\infty},L_\infty)$.}\label{sectionsigma}

As above, we will denote $\dddot{\b}(t)=\b(t^{\frac{1}{\varkappa}})$ for all $t>0$, $0<\varkappa<\infty$. The following theorem improves Corollary 5.3 from \cite{GOT_2005} and can be proved analogously using Lemma \ref{lemma15}.

\begin{thm}\label{corollary27}
Let $0<\theta\leq1$, $0<\varkappa<\infty$, $p=\frac{\varkappa}{1-\theta}$, $0<q\leq \infty$ and $\b\in SV$. Then,
$$ (L_\varkappa,L_\infty)_{\theta,q;\b}=(L_{\varkappa,\infty},L_\infty)_{\theta,q;\b}=L_{p,q;\dddot{\b}}.$$
\end{thm}

In the following subsections we study similar identities for the limiting and extremal constructions.
\subsection{${\mathcal L}$ and ${\mathcal L\mathcal L}$ spaces}\label{sub43}
Next theorem improves \cite[Lemma 5.4]{GOT_2005}.
\begin{thm}\label{corollary28}
Let $0<\theta\leq 1$, $0<\varkappa<\infty$, $p=\frac{\varkappa}{1-\theta}$, $0<q, r\leq \infty$ and $a, \b\in SV$. Then,
$$(L_\varkappa,L_\infty)^{\mathcal L}_{\theta,r,\b,q,a}=(L_{\varkappa,\infty},L_\infty)^{\mathcal L}_{\theta,r,\b,q,a}=L^{\mathcal L}_{p,r,\dddot{\b},q,\dddot{a}}.$$
In particular,
$$(L_\varkappa,L_\infty)^{\mathcal L}_{\theta,r,\b,q,1}=(L_{\varkappa,\infty},L_\infty)^{\mathcal L}_{\theta,r,\b,q,1}=L^{(p,q,r}_{\dddot{\b}}.$$
\end{thm}
\begin{thm}\label{corollary29}
Let $0<\theta\leq 1$, $0<\varkappa<\infty$, $p=\frac{\varkappa}{1-\theta}$, $0<q, r, s\leq \infty$ and $a, \b, c\in SV$. Then,
$$(L_\varkappa,L_\infty)^{\mathcal L,\mathcal L}_{\theta,s,c,r,\b,q,a}=(L_{\varkappa,\infty},L_\infty)^{\mathcal L,\mathcal L}_{\theta,s,c,r,\b,q,a}=L^{\mathcal L,\mathcal L}_{p,(s,\dddot{c},r,\dddot{\b},q,\dddot{a})}.$$
\end{thm}
\begin{proof}
Put $X=(A_0,L_\infty)^{\mathcal L,\mathcal L}_{\theta,s,c,r,\b,q,a}$, where $A_0=L_\varkappa$ or $A_0=L_{\varkappa,\infty}$, respectively. By Lemma \ref{lemma15} and the change of variables $y=u^\varkappa$, $x=t^\varkappa$, we obtain
\begin{align*}
\|f\|_{X}&\approx
\bigg\|t^{-\frac{1}{s}}c(t)\Big\|u^{-\frac{1}{r}}\b(u) \|v^{\frac{1}{p}-\frac{1}{q}} \dddot{a}(v) f^*(v) \|_{q,(0,u^\varkappa)}\Big\|_{r,(0,t)}\bigg\|_{s,(0,\infty)}\\
&\approx\bigg\|t^{-\frac{1}{s}}c(t)\Big\|y^{-\frac{1}{r}}\dddot{\b}(y) \|v^{\frac{1}{p}-\frac{1}{q}} \dddot{a}(v) f^*(v) \|_{q,(0,y)}\Big\|_{r,(0,t^\varkappa)}\bigg\|_{s,(0,\infty)}\\
&\approx\bigg\|x^{-\frac{1}{s}}\dddot{c}(x)\Big\|y^{-\frac{1}{r}}\dddot{\b}(y) \|v^{\frac{1}{p}-\frac{1}{q}} \dddot{a}(v) f^*(v) \|_{q,(0,y)}\Big\|_{r,(0,x)}\bigg\|_{s,(0,\infty)}=\|f\|_{L^{\mathcal L,\mathcal L}_{p,(s,\dddot{c},r,\dddot{\b},q,\dddot{a})}}.
\end{align*}
This completes the proof.
\end{proof}

\subsection{${\mathcal R}$ and ${\mathcal R\mathcal R}$ spaces}\label{sub44}

\begin{thm}\label{corollary30}
Let $0<\theta\leq 1$, $0<\varkappa<\infty$, $p=\tfrac{\varkappa}{1-\theta}$, $0<q, r\leq \infty$ and $a, \b\in SV$. Then,
$$(L_\varkappa,L_\infty)^{\mathcal R}_{\theta,r,\b,q,a}=(L_{\varkappa,\infty},L_\infty)^{\mathcal R}_{\theta,r,\b,q,a}=L^{\mathcal R}_{p,r,\dddot{\b},q,\dddot{a}}.$$
In particular,
$$(L_\varkappa,L_\infty)^{\mathcal R}_{\theta,r,\b,q,1}=(L_{\varkappa,\infty},L_\infty)^{\mathcal R}_{\theta,r,\b,q,1}=L^{p),q,r}_{\dddot{\b}}.$$
\end{thm}
\begin{proof}
Put $X=(A_0,L_\infty)^{\mathcal R}_{\theta,r,\b,q,a}$, where $A_0=L_\varkappa$ or $A_0=L_{\varkappa,\infty}$. Lemma \ref{lemma16} and the change of variables $x=t^\varkappa$ imply that
\begin{align*}
\|f\|_{X}&\gtrsim\Big\|t^{-\frac{1}{r}}\b(t)\big\|u^{\frac{1-\theta}{\varkappa}-\frac{1}{q}}\,\dddot{a}(u)f^{*}(u)\big\|_{q,(t^\varkappa,\infty)}\Big\|_{r,(0,\infty)}\\
&\approx\Big\|x^{-\frac{1}{r}}\dddot{\b}(x)\big\|u^{\frac{1}{p}-\frac{1}{q}}\,\dddot{a}(u)f^{*}(u)\big\|_{q,(x,\infty)}\Big\|_{r,(0,\infty)}=\|f\|_{L^{\mathcal R}_{p,r,\dddot{\b},q,\dddot{a}}}.
\end{align*}
On the other hand, by \eqref{eKsigma}, \eqref{ec99} and suitable change of variables, it follows that 
\begin{align*}
\|f\|_{X}&\lesssim\bigg\|t^{-\frac{1}{r}}\b(t)\Big\|u^{-\theta-\frac{1}{q}}a(u)\Big(\int_0^{u^\varkappa}f^*(\tau)^\varkappa\, d\tau\Big)^{1/\varkappa}\Big\|_{q,(t,\infty)}\bigg\|_{r,(0,\infty)}\\
&=\bigg\|t^{-\frac{1}{r}}\b(t)\Big\|u^{-\theta\varkappa-\frac{\varkappa}{q}}a(u)^\varkappa\int_0^{u^\varkappa}f^*(\tau)^\varkappa\, d\tau\Big\|^{1/\varkappa}_{q/\varkappa,(t,\infty)}\bigg\|_{r,(0,\infty)}\\
&\approx\bigg\|t^{-\frac{1}{r}}\b(t)\Big\|y^{-\theta-\frac{\varkappa}{q}}\,\dddot{a}(y)^\varkappa\int_0^{y}f^*(\tau)^\varkappa\, d\tau\Big\|^{1/\varkappa}_{q/\varkappa,(t^\varkappa,\infty)}\bigg\|_{r,(0,\infty)}\\
&\approx\bigg\|x^{-\frac{1}{r}}\dddot{\b}(x)\Big\|y^{-\theta-\frac{\varkappa}{q}}\,\dddot{a}(y)^\varkappa\int_0^{y}f^*(\tau)^\varkappa\, d\tau\Big\|^{1/\varkappa}_{q/\varkappa,(x,\infty)}\bigg\|_{r,(0,\infty)}\\
&=\bigg\|x^{-\frac{\varkappa}{r}}\dddot{\b}(x)^\varkappa\Big\|y^{-\theta-\frac{\varkappa}{q}}\,\dddot{a}(y)^\varkappa\int_0^{y}f^*(\tau)^\varkappa\, d\tau\Big\|_{q/\varkappa,(x,\infty)}\bigg\|^{1/\varkappa}_{r/\varkappa,(0,\infty)}.
\end{align*}
Now, the previous estimate and Corollary \ref{corollary6} give
\begin{align*}
\|f\|_{X}&\lesssim
\Big\|x^{-\frac{\varkappa}{r}}\dddot{\b}(x)^\varkappa\big\|y^{1-\theta-\frac{\varkappa}{q}}\,\dddot{a}(y)^\varkappa f^*(y)^\varkappa\big\|_{q/\varkappa,(x,\infty)}\Big\|^{1/\varkappa}_{r/\varkappa,(0,\infty)}\\
&= \Big\|x^{-\frac{1}{r}}\dddot{\b}(x)\big\|y^{1-\theta-\frac{\varkappa}{q}}\,\dddot{a}(y)^\varkappa f^*(y)^\varkappa\big\|^{1/\varkappa}_{q/\varkappa,(x,\infty)}\Big\|_{r,(0,\infty)}\\
&=\Big\|x^{-\frac{1}{r}}\dddot{\b}(x)\big\|y^{\frac{1-\theta}{\varkappa}-\frac{1}{q}}\,\dddot{a}(y)f^*(y)\big\|_{q,(x,\infty)}\Big\|_{r,(0,\infty)}=\|f\|_{L^{\mathcal R}_{p,r,\dddot{\b},q,\dddot{a}}}.
\end{align*}
This completes de proof.
\end{proof}

\begin{thm}\label{corollary31}
Let $0<\theta\leq 1$, $0<\varkappa<\infty$, $p=\frac{\varkappa}{1-\theta}$, $0<q,s,r\leq \infty$ and $a, \b, c\in SV$. Then,
$$(L_\varkappa,L_\infty)^{\mathcal R,\mathcal R}_{\theta,s,c,r,\b,q,a}=(L_{\varkappa,\infty},L_\infty)^{\mathcal R,\mathcal R}_{\theta,s,c,r,\b,q,a}=L^{\mathcal R,\mathcal R}_{p,(s,\dddot{c},r,\dddot{\b},q,\dddot{a})}.$$
\end{thm}

\begin{proof}
Put $X=(A_0,L_\infty)^{\mathcal R,\mathcal R}_{\theta,s,c,r,\b,q,a}$, where $A_0=L_\varkappa$ or $A_0=L_{\varkappa,\infty}$. Using Lemma \ref{lemma16} and the change of variables $u^m=y$, $t^m=x$ we obtain the desired lower bound
\begin{align*}
\|f\|_{X}&\gtrsim \bigg\|t^{-\frac{1}{s}}\,c(t)\Big\|u^{-\frac{1}{r}}\,\b(u)\big\|z^{\frac{1-\theta}{\varkappa}-\frac{1}{q}}\,\dddot{a}(z) f^*(z)\Big\|_{q,(u^m,\infty)}\Big\|_{r,(t,\infty)}\bigg\|_{s,(0,\infty)}\\
&\approx\bigg\|x^{-\frac{1}{s}}\,\dddot{c}(x)\Big\|y^{-\frac{1}{r}}\,\dddot{\b}(y)\big\|z^{\frac{1}{p}-\frac{1}{q}}\,\dddot{a}(z) f^*(z)\Big\|_{q,(y,\infty)}\Big\|_{r,(x,\infty)}\bigg\|_{s,(0,\infty)}\\&=\|f\|_{L^{\mathcal R,\mathcal R}_{p,(s,\dddot{c},r,\dddot{\b},q,\dddot{a})}}.
\end{align*}
Now we proceed with the upper bound. Using \eqref{eKsigma}, \eqref{ec99} and making the change of variables $v^\varkappa=z$, $u^\varkappa=y$ and $t^\varkappa=x$, we deduce that
\begin{align*}
\|f\|_{X}&\lesssim\Bigg\|t^{-\frac{1}{s}}c(t)\bigg\|u^{-\frac{1}{r}}\b(u)\Big\|v^{-\theta-\frac{1}{q}}a(v)\Big(\int_0^{v^\varkappa}f^*(\tau)^\varkappa\, d\tau\Big)^{1/\varkappa}\Big\|_{q,(u,\infty)}\bigg\|_{r,(t,\infty)}\Bigg\|_{s,(0,\infty)}\\
&=\Bigg\|t^{-\frac{\varkappa}{s}}c(t)^\varkappa\bigg\|u^{-\frac{\varkappa}{r}}\b(u)^\varkappa\Big\|v^{-\theta\varkappa-\frac{\varkappa}{q}}a(v)^\varkappa\int_0^{v^\varkappa}f^*(\tau)^\varkappa\, d\tau\Big\|_{\frac{q}{\varkappa},(u,\infty)}\bigg\|_{\frac{r}{\varkappa},(t,\infty)}\Bigg\|^{1/\varkappa}_{\frac{s}{\varkappa},(0,\infty)}\\
&\approx\Bigg\|x^{-\frac{1}{s}}\,\dddot{c}(x)^\varkappa\bigg\|y^{-\frac{1}{r}}\,\dddot{\b}(y)^\varkappa\Big\|z^{-\theta-\frac{\varkappa}{q}}\,\dddot{a}(z)^\varkappa\int_0^{z}f^*(\tau)^\varkappa\, d\tau\Big\|_{\frac{q}{\varkappa},(y,\infty)}\bigg\|_{\frac{r}{\varkappa},(x,\infty)}\Bigg\|^{1/\varkappa}_{\frac{s}{\varkappa},(0,\infty)}.
\end{align*}
Corollary \ref{corollary7} yields the estimate 
\begin{align*}
\|f\|_{X}&\lesssim \bigg\|x^{-\frac{1}{s}}\,\dddot{c}(x)^\varkappa\Big\|y^{-\frac{1}{r}}\,\dddot{\b}(y)^\varkappa\big\|z^{1-\theta-\frac{\varkappa}{q}}\,\dddot{a}(z)^\varkappa f^*(z)^\varkappa\Big\|_{\frac{q}{\varkappa},(y,\infty)}\Big\|_{\frac{r}{\varkappa},(x,\infty)}\bigg\|^{1/\varkappa}_{\frac{s}{\varkappa},(0,\infty)}\\
&=\bigg\|x^{-\frac{1}{s}}\,\dddot{c}(x)\Big\|y^{-\frac{1}{r}}\,\dddot{\b}(y)\big\|z^{\frac{1}{p}-\frac{1}{q}}\,\dddot{a}(z) f^*(z)\Big\|_{q,(y,\infty)}\Big\|_{r,(x,\infty)}\bigg\|_{s,(0,\infty)}\\&=\|f\|_{L^{\mathcal R,\mathcal R}_{p,(s,\dddot{c},r,\dddot{\b},q,\dddot{a})}}.
\end{align*}
\end{proof}

\subsection{${\mathcal L\mathcal R}$ and ${\mathcal R\mathcal L}$ spaces}\label{sub45}

\begin{thm}\label{corrollary33}
Let $0<\theta\leq 1$, $0<\varkappa<\infty$, $p=\frac{\varkappa}{1-\theta}$, $0<q, r, s\leq \infty$ and $a, \b, c\in SV$ with $\b$ satisfying $\big\|u^{-\frac{1}{r}}\b(u)\big\|_{r,(1,\infty)}<\infty$. Then,
$$(L_\varkappa,L_\infty)^{\mathcal L,\mathcal R}_{\theta,s,c,r,\b,q,a}=(L_{\varkappa,\infty},L_\infty)^{\mathcal L,\mathcal R}_{\theta,s,c,r,\b,q,a}=L^{\mathcal L,\mathcal R}_{p,(s,\dddot{c},r,\dddot{\b},q,\dddot{a})}.$$
\end{thm}
\begin{proof}
First note that the change of variables $y=u^\varkappa$ and $x=t^\varkappa$ shows that
\begin{align}
\|f\|_{L^{\mathcal L,\mathcal R}_{p,(s,\dddot{c},r,\dddot{\b},q,\dddot{a})}}&=\bigg\|x^{-\frac{1}{s}}\,\dddot{c}(x)\Big\|y^{-\frac{1}{r}}\,\dddot{\b}(y) \|z^{\frac{1}{p}-\frac{1}{q}}\,\dddot{a}(z)f^*(z)\|_{q,(x,y)}\Big\|_{r,(x,\infty)}\bigg\|_{s,(0,\infty)}\nonumber\\
&\approx\bigg\|t^{-\frac{1}{s}}c(t)\Big\|u^{-\frac{1}{r}}\b(u) \|z^{\frac{1}{p}-\frac{1}{q}}\, \dddot{a}(z)f^*(z)\|_{q,(t^\varkappa,u^\varkappa)}\Big\|_{r,(t,\infty)}\bigg\|_{s,(0,\infty)}.\label{ec15}
\end{align}
Put $X=(A_0,L_\infty)^{\mathcal L,\mathcal R}_{\theta,s,c,r,\b,q,a}$ where $A_0=L_\varkappa$ or $A_0=L_{\varkappa,\infty}$. By Lemma \ref{lemma16} and \eqref{ec15} we have that
\begin{align*}
\|f\|_{X}
&=\bigg\|t^{-\frac{1}{s}}c(t)\Big\|u^{-\frac{1}{r}}\b(u) \|v^{-\theta-\frac{1}{q}} a(v)K(v,f;A_0,L_\infty)\|_{q,(t,u)}\Big\|_{r,(t,\infty)}\bigg\|_{s,(0,\infty)}\\
&\gtrsim\bigg\|t^{-\frac{1}{s}}c(t)\Big\|u^{-\frac{1}{r}}\b(u) \|z^{\frac{1}{p}-\frac{1}{q}}\, \dddot{a}(z)f^*(z)\|_{q,(t^\varkappa,u^\varkappa)}\Big\|_{r,(t,\infty)}\bigg\|_{s,(0,\infty)}\approx\|f\|_{L^{\mathcal L,\mathcal R}_{p,(s,\dddot{c},r,\dddot{\b},q,\dddot{a})}}.
\end{align*}
Next, we proceed with the reverse estimate. Lemma \ref{lemma17}  implies that
\begin{align*}
\|f\|_{X}&\lesssim\bigg\|t^{-\frac{1}{s}}c(t)\Big\|u^{-\frac{1}{r}}\b(u) \Big(t^{\frac{\varkappa}{p}}a(t)f^{**}_{(\varkappa)}(t)+\|z^{\frac{1}{p}-\frac{1}{q}} \dddot{a}(z)f^*(z)\|_{q,(t^\varkappa,u^\varkappa)}\Big)\Big\|_{r,(t,\infty)}\bigg\|_{s,(0,\infty)}\\
&\leq \bigg\|t^{\frac{\varkappa}{p}-\frac{1}{s}}c(t)a(t)f^{**}_{(\varkappa)}(t)\Big\|u^{-\frac{1}{r}}\b(u) \Big\|_{r,(t,\infty)}\bigg\|_{s,(0,\infty)}\\
&+\bigg\|t^{-\frac{1}{s}}c(t)\Big\|u^{-\frac{1}{r}}\b(u) \|z^{\frac{1}{p}-\frac{1}{q}} \dddot{a}(z)f^*(z)\|_{q,(t^\varkappa,u^\varkappa)}\Big)\Big\|_{r,(t,\infty)}\bigg\|_{s,(0,\infty)}=S_1+S_2.
\end{align*}
By \eqref{ec15} we have that $S_2\approx\|f\|_{L^{\mathcal L,\mathcal R}_{p,(s,\dddot{c},r,\dddot{\b},q,\dddot{a})}}$, so to finish the proof it suffices to establish that $S_1\lesssim S_2$. Indeed, by Lemma \ref{lemma15} and Lemma \ref{lemma3}  and suitable change of variables, we have
\begin{align*}
S_1&=\bigg\|t^{\frac{\varkappa}{p}-\frac{1}{s}}c(t)a(t)f^{**}_{(\varkappa)}(t)\big\|u^{-\frac{1}{r}}\b(u) \big\|_{r,(t,\infty)}\bigg\|_{s,(0,\infty)}\\
&\approx\bigg\|x^{\frac{1}{p}-\frac{1}{s}}\,\dddot{c}(x)\dddot{a}(x)f^{*}(x)\big\|u^{-\frac{1}{r}}\b(u) \big\|_{r,(x^{\frac{1}{\varkappa}},\infty)}\bigg\|_{s,(0,\infty)}\\
&\approx\bigg\|x^{\frac{1}{p}-\frac{1}{s}}\,\dddot{c}(x)\dddot{a}(x)f^{*}(x)\big\|y^{-\frac{1}{r}}\dddot{\b}(y) \big\|_{r,(x,\infty)}\bigg\|_{s,(0,\infty)}\\
&\lesssim \bigg\|x^{-\frac{1}{s}}\,\dddot{c}(x)\big\|y^{-\frac{1}{r}}\dddot{\b}(y) \big\|_{r,(x,\infty)}\big\|z^{\frac{1}{p}-\frac{1}{q}}\,\dddot{a}(z)f^{*}(z)\big\|_{q,(x/2,x)}\bigg\|_{s,(0,\infty)}\\
&\leq \bigg\|x^{-\frac{1}{s}}\,\dddot{c}(x)\Big\|y^{-\frac{1}{r}}\dddot{\b}(y) \big\|z^{\frac{1}{p}-\frac{1}{q}}\,\dddot{a}(z)f^{*}(z)\big\|_{q,(x/2,y)}\Big\|_{r,(x,\infty)}\bigg\|_{s,(0,\infty)}\\
&\leq \bigg\|x^{-\frac{1}{s}}\,\dddot{c}(x)\Big\|y^{-\frac{1}{r}}\dddot{\b}(y) \big\|z^{\frac{1}{p}-\frac{1}{q}}\,\dddot{a}(z)f^{*}(z)\big\|_{q,(x/2,y)}\Big\|_{r,(x/2,\infty)}\bigg\|_{s,(0,\infty)}\\
&\approx \bigg\|x^{-\frac{1}{s}}\,\dddot{c}(x)\Big\|y^{-\frac{1}{r}}\dddot{\b}(y) \big\|z^{\frac{1}{p}-\frac{1}{q}}\,\dddot{a}(z)f^{*}(z)\big\|_{q,(x,y)}\Big\|_{r,(x,\infty)}\bigg\|_{s,(0,\infty)}=\|f\|_{L^{\mathcal L,\mathcal R}_{p,(s,\dddot{c},r,\dddot{\b},q,\dddot{a})}}.
\end{align*}
The proof is completed.
\end{proof}

\begin{thm}\label{corollary32}
Let $0<\theta\leq 1$, $0<\varkappa<\infty$, $p=\frac{\varkappa}{1-\theta}$, $0<q, r, s\leq \infty$ and $a, \b, c\in SV$. Then,
$$
(L_\varkappa,L_\infty)^{\mathcal R,\mathcal L}_{\theta,s,c,r,\b,q,a}=(L_{\varkappa,\infty},L_\infty)^{\mathcal R,\mathcal L}_{\theta,s,c,r,\b,q,a}=L^{\mathcal R,\mathcal L}_{p,(s,\dddot{c},r,\dddot{\b},q,\dddot{a})}.$$
\end{thm}
\begin{proof}
In this case the changes of variables $y=u^\varkappa$ and $x=t^\varkappa$ show that
\begin{align}
\|f\|_{L^{\mathcal R,\mathcal L}_{p,(s,\dddot{c},r,\dddot{\b},q,\dddot{a})}}&=\bigg\|x^{-\frac{1}{s}}\dddot{c}(x)\Big\|y^{-\frac{1}{r}}\dddot{\b}(y) \|z^{\frac{1}{p}-\frac{1}{q}} \dddot{a}(z)f^*(z)\|_{q,(y,x)}\Big\|_{r,(0,x)}\bigg\|_{s,(0,\infty)}\nonumber\\
&\approx\bigg\|t^{-\frac{1}{s}}c(t)\Big\|u^{-\frac{1}{r}}\b(u) \|z^{\frac{1}{p}-\frac{1}{q}} \dddot{a}(z)f^*(z)\|_{q,(u^\varkappa,t^\varkappa)}\Big\|_{r,(0,t)}\bigg\|_{s,(0,\infty)}.\label{ec14}
\end{align}
Put $X=(A_0,L_\infty)^{\mathcal R,\mathcal L}_{\theta,s,c,r,\b,q,a}$, where $A_0=L_\varkappa$ or $A_0=L_{\varkappa,\infty}$. By Lemma \ref{lemma16} and \eqref{ec14} we have that
\begin{align*}
\|f\|_{X}&=\bigg\|t^{-\frac{1}{s}}c(t)\Big\|u^{-\frac{1}{r}}\b(u) \|v^{-\theta-\frac{1}{q}} a(v)K(v,f;A_0,L_\infty)\|_{q,(u,t)}\Big\|_{r,(0,t)}\bigg\|_{s,(0,\infty)}\\
&\gtrsim\bigg\|t^{-\frac{1}{s}}c(t)\Big\|u^{-\frac{1}{r}}\b(u) \|z^{\frac{1}{p}-\frac{1}{q}} \dddot{a}(z)f^*(z)\|_{q,(u^\varkappa,t^\varkappa)}\Big\|_{r,(0,t)}\bigg\|_{s,(0,\infty)}=\|f\|_{L^{\mathcal R,\mathcal L}_{p,(s,\dddot{c},r,\dddot{\b},q,\dddot{a})}}.
\end{align*}
Next, we prove the reverse estimate. Lemma \ref{lemma17} yields that
\begin{align*}
\|f\|_{X}&\lesssim\bigg\|t^{-\frac{1}{s}}c(t)\Big\|u^{-\frac{1}{r}}\b(u) \Big(u^{\frac{\varkappa}{p}}a(u)f^{**}_{(\varkappa)}(u)+\|z^{\frac{1}{p}-\frac{1}{q}} \dddot{a}(z)f^*(z)\|_{q,(u^\varkappa,t^\varkappa)}\Big)\Big\|_{r,(0,t)}\bigg\|_{s,(0,\infty)}\\
&\leq \bigg\|t^{-\frac{1}{s}}c(t)\Big\|u^{\frac{\varkappa}{p}-\frac{1}{r}}\b(u) a(u)f^{**}_{(\varkappa)}(u)\Big\|_{r,(0,t)}\bigg\|_{s,(0,\infty)}\\
&+\bigg\|t^{-\frac{1}{s}}c(t)\Big\|u^{-\frac{1}{r}}\b(u) \|z^{\frac{1}{p}-\frac{1}{q}} \dddot{a}(z)f^*(z)\|_{q,(u^\varkappa,t^\varkappa)}\Big)\Big\|_{r,(0,t)}\bigg\|_{s,(0,\infty)}=T_1+T_2.
\end{align*}
By \eqref{ec14} we have that $T_2\approx\|f\|_{L^{\mathcal R,\mathcal L}_{p,(s,\dddot{c},r,\dddot{\b},q,\dddot{a})}}$, so to finish the proof it is enough to obtain that $T_1\lesssim T_2$. Indeed, by Lemma \ref{lemma15} and Lemma \ref{lemma3} we have
\begin{align*}
T_1&:=\bigg\|t^{-\frac{1}{s}}c(t)\Big\|u^{\frac{\varkappa}{p}-\frac{1}{r}}\b(u) a(u)f^{**}_{(\varkappa)}(u)\Big\|_{r,(0,t)}\bigg\|_{s,(0,\infty)}\\
&\approx\bigg\|t^{-\frac{1}{s}}c(t)\Big\|y^{\frac{1}{p}-\frac{1}{r}}\dddot{\b}(y)\dddot{a}(y)f^{*}(y)\Big\|_{r,(0,t^\varkappa)}\bigg\|_{s,(0,\infty)}\\
&\approx\bigg\|t^{-\frac{1}{s}}c(t)\Big\|y^{\frac{1}{p}-\frac{1}{r}}\dddot{\b}(y)\dddot{a}(y)f^{*}(2y)\Big\|_{r,(0,t^\varkappa/2)}\bigg\|_{s,(0,\infty)}\\
&\lesssim\bigg\|t^{-\frac{1}{s}}c(t)\Big\|y^{-\frac{1}{r}}\dddot{\b}(y)\| z^{\frac{1}{p}-\frac{1}{q}}\dddot{a}(z)f^{*}(z)\|_{q,(y,2y)}\Big\|_{r,(0,t^\varkappa/2)}\bigg\|_{s,(0,\infty)}\\
&\leq\bigg\|t^{-\frac{1}{s}}c(t)\Big\|y^{-\frac{1}{r}}\dddot{\b}(y)\| z^{\frac{1}{p}-\frac{1}{q}}\dddot{a}(z)f^{*}(z)\|_{q,(y,t^\varkappa)}\Big\|_{r,(0,t^\varkappa/2)}\bigg\|_{s,(0,\infty)}\\
&\leq\bigg\|t^{-\frac{1}{s}}c(t)\Big\|y^{-\frac{1}{r}}\dddot{\b}(y)\| z^{\frac{1}{p}-\frac{1}{q}}\dddot{a}(z)f^{*}(z)\|_{q,(y,t^\varkappa)}\Big\|_{r,(0,t^\varkappa)}\bigg\|_{s,(0,\infty)}\\
&\approx\bigg\|x^{-\frac{1}{s}}\dddot{c}(x)\Big\|y^{-\frac{1}{r}}\dddot{\b}(y)\| z^{\frac{1}{p}-\frac{1}{q}}\dddot{a}(z)f^{*}(z)\|_{q,(y,x)}\Big\|_{r,(0,x)}\bigg\|_{s,(0,\infty)}=\|f\|_{L^{\mathcal R,\mathcal L}_{p,s,\dddot{c},r,\dddot{\b},q,\dddot{a}}}.
\end{align*}
This completes the proof.
\end{proof}

\section{Interpolation formulae for the couples $(L_1,L_\infty)$ and $(L_{1,\infty},L_\infty)$}\label{section6}

Finally, taking $\varkappa=1$ in the results of Section \ref{sectionsigma}, we are in position to formulate the main theorems of the paper. First, as a direct consequence of Lemma \ref{lemma15} we obtain the following corollary:

\begin{cor}\label{corollary34}
Let $0<\theta\leq1$, $p=\frac{1}{1-\theta}$, $0<q\leq \infty$ and $a\in SV$. Then, for all $f\in L_1+L_\infty$ and all $t>0$
$$\Big\|u^{-\theta-\frac{1}{q}}a(u)K(u,f;L_1,L_\infty)\Big\|_{q,(0,t)}=\Big\|u^{\frac{1}{p}-\frac{1}{q}}a(u)f^{**}(u)\Big\|_{q,(0,t)}\approx \Big\|v^{\frac{1}{p}-\frac{1}{q}}a(v)f^{*}(v)\Big\|_{q,(0,t)}$$
and for all $f\in L_{1,\infty}+L_\infty$ and all $t>0$
$$\Big\|u^{-\theta-\frac{1}{q}}a(u)K(u,f;L_{1,\infty},L_\infty)\Big\|_{q,(0,t)}\approx \Big\|v^{\frac{1}{p}-\frac{1}{q}}a(v)f^{*}(v)\Big\|_{q,(0,t)}.$$
\end{cor}
Theorems \ref{corollary27} to \ref{corollary32}, Corollary \ref{corollary34} and Lemmas \ref{lemma24}, \ref{lemma25} imply the following results.

\begin{thm} (Cf. \cite{FS_2015}, \cite[Proposition 3]{Y_1969}, \cite[Theorem 3.15]{P_2021}).
Let $0<\theta\leq1$, $p=\frac{1}{1-\theta}$, $0<q\leq \infty$ and $a\in SV$. Then,
$$ (L_1,L_\infty)_{\theta,q;a}=(L_{1,\infty},L_\infty)_{\theta,q;a}=L_{(p,q;a)}=L_{p,q;a}.$$
\end{thm}

\begin{thm}\label{teo63}
Let $0<\theta\leq 1$, $p=\tfrac{1}{1-\theta}$, $0<q, r\leq \infty$ and $a, \b\in SV$. Then,
$$(L_1,L_\infty)^{\mathcal L}_{\theta,r,\b,q,a}=(L_{1,\infty},L_\infty)^{\mathcal L}_{\theta,r,\b,q,a}=(L)^{\mathcal L}_{p,r,\b,q,a}=L^{\mathcal L}_{p,r,\b,q,a}.$$
In particular, (Cf. \cite[Corollary 3.3]{FK_2004}),
$$(L_1,L_\infty)^{\mathcal L}_{\theta,r,\b,q,1}=(L_{1,\infty},L_\infty)^{\mathcal L}_{\theta,r,\b,q,1}=(L)^{(p,q,r}_{\b}=L^{(p,q,r}_{\b}.$$
\end{thm}

\begin{thm}\label{teo64}
Let $0<\theta\leq 1$, $p=\frac{1}{1-\theta}$, $0<q, r\leq \infty$ and $a, \b\in SV$. Then,
$$(L_1,L_\infty)^{\mathcal R}_{\theta,r,\b,q,a}=(L_{1,\infty},L_\infty)^{\mathcal R}_{\theta,r,\b,q,a}=(L)^{\mathcal R}_{p,r,\b,q,a}=L^{\mathcal R}_{p,r,\b,q,a}.$$
In particular, (Cf. \cite[Theorem 4.2]{FK_2004}), 
$$(L_1,L_\infty)^{\mathcal R}_{\theta,r,\b,q,1}=(L_{1,\infty},L_\infty)^{\mathcal R}_{\theta,r,\b,q,1}=(L)^{p),q,r}_{\b}=L_{\b}^{p),q,r}.$$
\end{thm}

\begin{thm}
Let $0<\theta\leq 1$, $p=\frac{1}{1-\theta}$, $0<q, r, s\leq \infty$ and $a, b, c\in SV$. Then,
$$(L_1,L_\infty)^{\mathcal L,\mathcal L}_{\theta,s,c,r,\b,q,a}=(L_{1,\infty},L_\infty)^{\mathcal L,\mathcal L}_{\theta,s,c,r,\b,q,a}=(L)^{\mathcal L,\mathcal L}_{p,(s,c,r,\b,q,a)}=L^{\mathcal L,\mathcal L}_{p,(s,c,r,\b,q,a)}$$
and
$$(L_1,L_\infty)^{\mathcal R,\mathcal R}_{\theta,s,c,r,\b,q,a}=(L_{1,\infty},L_\infty)^{\mathcal R,\mathcal R}_{\theta,s,c,r,\b,q,a}=(L)^{\mathcal R,\mathcal R}_{p,(s,c,r,\b,q,a)}=L^{\mathcal R,\mathcal R}_{p,(s,c,r,\b,q,a)}.$$
\end{thm}

\begin{thm}
Let $0<\theta\leq 1$, $p=\tfrac{1}{1-\theta}$, $0<q, r, s\leq \infty$ and $a, b, c\in SV$. Then,
$$(L_1,L_\infty)^{\mathcal R,\mathcal L}_{\theta,s,c,r,\b,q,a}=(L_{1,\infty},L_\infty)^{\mathcal R,\mathcal L}_{\theta,s,c,r,\b,q,a}=(L)^{\mathcal R,\mathcal L}_{p,(s,c,r,\b,q,a)}=L^{\mathcal R,\mathcal L}_{p,(s,c,r,\b,q,a)}$$
and
$$(L_1,L_\infty)^{\mathcal L,\mathcal R}_{\theta,s,c,r,\b,q,a}=(L_{1,\infty},L_\infty)^{\mathcal L,\mathcal R}_{\theta,s,c,r,\b,q,a}=(L)^{\mathcal L,\mathcal R}_{p,(s,c,r,\b,q,a)}=L^{\mathcal L,\mathcal R}_{p,(s,c,r,\b,q,a)}.$$
\end{thm}

\section{Applications}\label{applications}
Here, we demostrate how our general assertions can be used to establish interpolation results for the grand and small Lorentz-Karamata spaces. For the sake of shortness, we present only one example.

\begin{cor}(Cf. \cite[Theorem 4.6]{FS_2021A}.)
Let $0<p_0<\infty$, $0<r, r_0, q_0\leq \infty$ and $a, a_0, \b_0\in SV$, with $\b_0$ satisfying $\|u^{-1/r_0}\b_0(u)\|_{r_0,(0,1)}<\infty$. Let $\chi(t)=a_0(t)\|u^{-1/r_0}\b_0(u)\|_{r_0,(0,t)}$ and $\rho(t)=t^{1/p_0}\chi(t)$, $t>0$.
\begin{itemize}
\item[(i)] If $\|u^{-1/r}a(u)\|_{r,(1,\infty)}<\infty$, then
$$\big(L_{p_0,r_0,\b_0,q_0,a_0}^{\mathcal R},L_\infty\big)_{0,r;a}=L_{p_0,(r,a\circ\rho,r_0,\b_0,q_0,a_0)}^{\mathcal R,\mathcal L}.$$
\item[(ii)] If $0<\theta<1$, then
$$\big(L_{p_0,r_0,\b_0,q_0,a_0}^{\mathcal R},L_\infty\big)_{\theta,r;a}=L_{p,r;a^\#},$$
where $\tfrac{1}{p}=\tfrac{1-\theta}{p_0}$ and $a^{\#}=\chi^{1-\theta}a\circ\rho$.
\item[(iii)] If $\|u^{-1/r}a(u)\|_{r,(0,1)}<\infty$, then
$$\big(L_{p_0,r_0,\b_0,q_0,a_0}^{\mathcal R},L_\infty\big)_{1,r;a}=L_{\infty,r;a\circ\rho}.$$
\end{itemize}
\end{cor}
\begin{proof}
Choose some $\varkappa$ such that $0<\varkappa<\min(p_0,q_0)$. Put $\theta_0=1-\frac{\varkappa}{p_0}$ and $X=\big(L_{p_0,r_0,\b_0,q_0,a_0}^{\mathcal R},L_\infty\big)_{\theta,r;a}$. By Theorem \ref{corollary30} we know that
$$L^{\mathcal R}_{p_0,r_0,\b_0,q_0,a_0}=(L_\varkappa,L_\infty)^{\mathcal R}_{\theta_0,r_0,\b_0(t^\varkappa),q_0,a_0(t^\varkappa)}.$$
Thus,
$$ X=\big((L_\varkappa,L_\infty)^{\mathcal R}_{\theta_0,r_0,\b_0(t^\varkappa),q_0,a_0(t^\varkappa)},L_\infty\Big)_{\theta,r;a}.$$
Case $0<\theta\leq 1$. By \cite[Theorem 23]{D_2020B} and Theorem \ref{corollary27}, we have
$$X=(L_\varkappa,L_\infty)_{\eta,r;y}=L_{p,r;x},$$
where $\eta=(1-\theta)\theta_0+\theta$, $\frac{1}{p}=\frac{1-\eta}{\varkappa}$,
$$y(t)=\Big(a_0(t^\varkappa)\|u^{-1/r_0}\b_0(u^\varkappa)\|_{r_0,(0,t)}\Big)^{1-\theta} a\Big(t^{1-\theta_0}a_0(t^{\varkappa})\|u^{-1/r_0}\b_0(u^\varkappa)\|_{r_0,(0,t)}\Big)$$
and 
$$x(t)=y(t^{\frac{1}{\varkappa}})=\Big(a_0(t)\|u^{-1/r_0}\b_0(u^\varkappa)\|_{r_0,(0,t^{\frac{1}{\varkappa}})}\Big)^{1-\theta} a\Big(t^{\frac{1-\theta_0}{\varkappa}}a_0(t)\|u^{-1/r_0}\b_0(u^\varkappa)\|_{r_0,(0,t^{\frac{1}{\varkappa}})}\Big).$$
Note that $\frac{1-\theta_0}{\varkappa}=\frac{1}{p_0}$, $\frac{1-\eta}{\varkappa}=\frac{1-\theta}{p_0}$, and $a_0(t)\|u^{-1/r_0}\b_0(u^\varkappa)\|_{r_0,(0,t^{\frac{1}{\varkappa}})}\approx\chi(t)$. Thus, $\frac{1}{p}=\frac{1-\theta}{p_0}$ and $x(t)\approx \chi(t)^{1-\theta}a\big(t^{1/p_0}\chi(t)\big)= a^\#(t)$.

Case $\theta=0$. By \cite[Theorem 28]{D_2021A} and Theorem \ref{corollary32}, we arrive at
$$X=(L_\varkappa,L_\infty)^{\mathcal R,\mathcal L}_{\theta_0,r,x,r_0,\b_0(t^\varkappa),q_0,a_0(t^\varkappa)}=
L^{\mathcal R,\mathcal L}_{p_0,(r,\overline{x},r_0,\b_0,q_0,a_0)}$$
where
$$x(t)=a\Big(t^{1-\theta_0}a_0(t^{\varkappa})\|u^{-1/r_0}\b_0(u^\varkappa)\|_{r_0,(0,t)}\Big)$$
and hence,
$$\dddot{x}(t)=a\Big(t^{\frac{1-\theta_0}{\varkappa}}a_0(t)\|u^{-1/r_0}\b_0(u^\varkappa)\|_{r_0,(0,t^{1/\varkappa})}\Big)\approx a\big(t^{\frac{1}{p_0}}\chi(t)\big)=(a\circ\rho)(t).$$
This completes the proof.
\end{proof}

In view of Remark \ref{remark21}, we finish with the following statement.

\begin{cor}\label{cor72}
(Cf. \cite[Corollary 5.9]{FS_2021A}, \cite[Corollary 62]{D_2021A}.)
Let $0<p_0<\infty$, $0<r, r_0, q_0\leq \infty$ and $a, \b_0\in SV$, with $\b_0$ satisfying $\|u^{-1/r_0}\b_0(u)\|_{r_0,(0,1)}<\infty$. Let $\chi(t)=\|u^{-1/r_0}\b_0(u)\|_{r_0,(0,t)}$ and $\rho(t)=t^{1/p_0}\chi(t)$, $t>0$.
\begin{itemize}
\item[(i)] If $\|u^{-1/r}a(u)\|_{r,(1,\infty)}<\infty$, then
$$\big(L^{p_0),q_0,r_0}_{\b_0},L_\infty\big)_{0,r;a}=L_{p_0,(r,a\circ\rho,r_0,\b_0,q_0,1)}^{\mathcal R,\mathcal L}.$$
\item[(ii)] If $0<\theta<1$, then
$$\big(L^{p_0),q_0,r_0}_{\b_0},L_\infty\big)_{\theta,r;a}=L_{p,r;a^\#},$$
where $\tfrac{1}{p}=\tfrac{1-\theta}{p_0}$ and $a^{\#}=\chi^{1-\theta}a\circ\rho$.
\item[(iii)] If $\|u^{-1/r}a(u)\|_{r,(0,1)}<\infty$, then
$$\big(L^{p_0),q_0,r_0}_{\b_0},L_\infty\big)_{1,r;a}=L_{\infty,r;a\circ\rho}.$$
\end{itemize}
\end{cor}

{\small \hspace{-5mm}{\textbf{Acknowledgments.}}
The authors Pedro Fern\'andez-Mart\'inez and Teresa Signes have been partially supported by grant MTM2017-84058-P (AEI/FEDER, UE).}

\end{document}